\documentclass[12pt]{amsart}

\usepackage{amssymb,latexsym}

\usepackage{enumerate}

\makeatletter

\@namedef{subjclassname@2010}{

  \textup{2010} Mathematics Subject Classification}

\makeatother
\newtheorem{thm}{Theorem}[section]
\newtheorem{cor}[thm]{Corollary}
\newtheorem{lem}[thm]{Lemma}
\newtheorem{pro}[thm]{Proposition}

\theoremstyle{definition}

\numberwithin{equation}{section}

\newcommand{\X}{\bold{X}}
\newcommand{\ex}{\Bbb E}
\newcommand{\da}{\mathcal{D}}

\newcommand{\pr}{\text{Prob}}
\newcommand{\sums}{\sideset{}{^\flat}\sum}
\newcommand{\sumh}{\sideset{}{^h}\sum}

\begin{document}

\baselineskip=17pt

\title[Distribution of large values of Zeta and $L$-functions] {On the distribution of extreme values of zeta and $L$-functions in the strip $\frac{1}{2}<\sigma<1$}

\author[Youness Lamzouri]{Youness Lamzouri}

\address{School of Mathematics, Institute for Advanced Study,
1 Einstein Drive,
Princeton, NJ, 08540
USA}

\email{lamzouri@ias.edu}

\date{}

\begin{abstract}

We study the distribution of large (and small) values of several families of
 $L$-functions on a line $\text{Re(s)}=\sigma$ where $1/2<\sigma<1$. We consider the Riemann
 zeta function $\zeta(s)$ in the $t$-aspect, Dirichlet $L$-functions in the $q$-aspect, and
 $L$-functions attached to primitive holomorphic cusp forms of weight $2$ in the level aspect.
 For each family we show that the $L$-values can be very well modeled by an adequate random Euler product,
uniformly in a wide range. We also prove new $\Omega$-results for quadratic Dirichlet $L$-functions
(predicted to be best possible by the probabilistic model) conditionally on GRH, and other results related to large moments of $\zeta(\sigma+it)$.
\end{abstract}


\thanks{The author is supported by a postdoctoral fellowship from the Natural
Sciences and Engineering Research Council of Canada and by the Institute
for Advanced Study and the National Science Foundation under agreement No. DMS-0635607.}

\maketitle

\section{Introduction and statement of results}

The analytic theory of $L$-functions has become a central part of modern number theory due to its diverse connections
 to several arithmetic, algebraic and geometric objects. The simplest example of an $L$-function is the Riemann zeta
function $\zeta(s)$ which  plays a fundamental role in the distribution of prime numbers. The study of the distribution
of values of $L$-functions has begun with the work of Bohr in the early twentieth century who established, using his theory
of almost periodic functions, that $\zeta(s)$ takes any non-zero complex value $c$ infinitely often in any strip $1<\text{Re}(s)<1+\epsilon$.
 Later in [1], Bohr refined his ideas by using probabilistic methods and, together with Jessen, showed that $\log \zeta(\sigma+it)$ has a
 continuous limiting distribution on the complex plane for any $\sigma>1/2$.

Let $1/2<\sigma<1$. The Riemann Hypothesis RH implies that for any $t\geq 3$ we have (see [27])
\begin{equation}
\log \zeta(\sigma+it)\ll
(\log t)^{2-2\sigma}/\log\log t.
\end{equation}
On the other hand,  Montgomery [19] showed that for $T$ large, we have
\begin{equation}
 \max_{t\in [T,2T]}\log |\zeta(\sigma+it)|\geq c (\log T)^{1-\sigma}(\log\log T)^{-\sigma},
\end{equation}
where $c=(\sigma-1/2)^{\frac12}/20$ unconditionally and $c=1/20$ on the assumption of RH. Moreover, based on a probabilistic
 argument, he conjectured that this result is likely to be best possible, more precisely that the true order of magnitude of
$\max_{t\in [T,2T]}\log|\zeta(\sigma+it)|$ corresponds to the $\Omega$-result (1.2) rather than the $O$-result (1.1). An important
motivation of our work is to investigate this question, and indeed the uniformity of our results supports Montgomery's conjecture. Define
$$ \Phi_T(\sigma,\tau):=\frac{1}{T}\text{meas}\{t\in [T,2T]: \log|\zeta(\sigma+it)|\geq \tau\}.$$

On the critical line $\sigma=1/2$, a wonderful result of Selberg (see [23] and [24]) states that as $t$ varies in $[T, 2T]$ the
distribution of $\log|\zeta(1/2+it)|$ is approximately
Gaussian with mean $0$ and variance $\frac1
2 \log \log T$. More precisely for any $\lambda\in \Bbb{R}$ we have
$$ \Phi_T\left(1/2,\lambda\sqrt{\frac1
2 \log \log T}\right)= \frac{1}{\sqrt{2\pi}}\int_{\lambda}^{\infty}e^{-x^2/2}dx+o(1).$$
Assuming RH, Soundararajan [26] has recently proved non-trivial upper bounds for $\Phi_T(1/2,\tau)$ in the range
 $\tau/\log\log T\to \infty$, which allows him to deduce upper bounds for the moments of $\zeta(1/2+it)$, nearly
of the conjectured order of magnitude.

On the edge of the critical strip (that is the line $\sigma=1$) the situation is more understood due to the facts
 that $\zeta(s)$ has an Euler product, and that its moments can be computed. In this case the RH implies that
$\log|\zeta(1+it)|\leq \log_3 t+\gamma+\log 2+o(1)$ (here and throughout
$\log_jx$ is the $j$-th iterate of the
natural logarithm). On the other direction the $\Omega$-result of Littlewood implies that
$\max_{[T,2T]}\log|\zeta(1+it)|\geq \log_3T+\gamma+o(1)$. In [9],  Granville and  Soundararajan studied
the behavior of the tail $\Phi_T(1,\tau)$, showing that uniformly for $\tau\leq \log_3T+\gamma-\epsilon$ we have
\begin{equation}
 \Phi_T(1,\tau)=\exp\left(-\frac{\exp(e^{\tau-\gamma}-a_0)}{e^{\tau-\gamma}}(1+o(1))\right),
\end{equation}
where $a_0$ is an explicit constant which is related to the probabilistic random Euler product they used to
 model the values $\zeta(1+it)$.

For $1/2<\sigma<1$, a consequence of Bohr and Jessen's work is that for $\tau\in \Bbb{R}$ we have that
$$\lim_{T\to\infty}\Phi_T(\sigma,\tau)=f(\sigma,\tau),$$ where $f(\sigma,\tau)$ is the tail of a continuous
distribution. Moreover it follows from the work of Montgomery and Odlyzko [20] that there exist $b_1,b_2>0$
 such that for  $\tau$
large
\begin{equation}
 \exp\left(-b_1\tau^{\frac{1}{1-\sigma}}(\log\tau)^{\frac{\sigma}{1-\sigma}}\right)\leq f(\sigma,\tau)
 \leq \exp\left(-b_2\tau^{\frac{1}{1-\sigma}}(\log\tau)^{\frac{\sigma}{1-\sigma}}\right).
\end{equation}

Our Theorem 1.1 estimates the function $\Phi_T(\sigma,\tau)$ uniformly for $\tau$ in a slightly smaller range than
the conjectured one, namely for $\tau\leq c_1(\sigma)(\log T)^{1-\sigma}/\log_2 T$ (for some suitably small
constant $c_1(\sigma)>0$), and shows that it decays precisely as in (1.4) in this range. Furthermore if this
result were to persist to the end of the viable range then this would imply Montgomery's conjecture.
The method is essentially an extension of the ideas of Granville and Soundararajan from [8] and [9]. Indeed
as in [8], [9], [14] and [15], the main idea is to compare the distribution of values of $\zeta(\sigma+it)$
with an adequate probabilistic model, defined as the random Euler product
$\zeta(\sigma,\X):=\prod_{p}\left(1-X(p)/p^{\sigma}\right)^{-1}$ (which converges a.s  if $\sigma>\frac12$)
where $\{X(p)\}_p$ are independent random variables uniformly distributed on the unit circle.

 Before describing our results let us first define some notation.  Let $X$ be a bounded real valued random
variable with $\ex(X)=0$ (here and throughout $\ex(\cdot)$ denotes the expectation). Then for any  $1/2<\sigma<1$
 we define
\begin{equation}
G_X(\sigma):= \int_{0}^{\infty}\frac{\log \ex(e^{uX})}{u^{\frac{1}{\sigma}+1}}du, \text{ and }
A_X(\sigma):=\left(\frac{\sigma^{2\sigma}}{(1-\sigma)^{2\sigma-1}G_X(\sigma)^{\sigma}}\right)^{\frac{1}{1-\sigma}},
\end{equation}
 (note that $G_X(\sigma)$ is absolutely convergent by Lemma 3.1 below). Moreover for $y,\tau\geq 3$
let
\begin{equation}
 r(y,\tau):= \left(\frac{\tau}{y^{1-\sigma}(\log y)^{-1}}\right)^{(\sigma-\frac{1}{2})/(1-\sigma)}.
\end{equation}
Then we prove
\begin{thm} Let $1/2<\sigma<1$, and $T$ be large. Then there exists $c_1(\sigma)>0$ such that uniformly
 in the range $1\ll \tau\leq c_1(\sigma)(\log T)^{1-\sigma}/\log_2T$ we have
$$ \Phi_{T}(\sigma,\tau)=\exp\left(-A_1(\sigma)\tau^{\frac{1}{(1-\sigma)}}(\log\tau)^{\frac{\sigma}{(1-\sigma)}}
\left(1+O\left(\frac{1}{\sqrt{\log\tau}}+r(\log T,\tau)\right)\right)\right),
$$
where $A_1(\sigma)=A_X(\sigma)$ with $X$ being a random variable uniformly distributed on the unit circle.
In this case we should note that  $\ex(e^{uX})=I_0(u):=\sum_{n=0}^{\infty}
(u/2)^{2n}/n!^2$ is the modified Bessel function of order $0$.
\end{thm}
{\it Remark 1.} This result is also proved for the distribution of large values of $\arg\zeta(\sigma+it)$ in
the same range of Theorem 1.1. Moreover, the same asymptotic does also hold for the left tail of the distribution
 of $\log|\zeta(\sigma+it)|$, (and also that of $\arg \zeta(\sigma+it)$) which is defined as the normalized
 measure of points $t\in [T,2T]$ such that $\log|\zeta(\sigma+it)|\leq -\tau$, in the same range of Theorem 1.1.

In general, in order to understand large values of $L$-functions it is often useful to consider high moments.
For $z$ a complex number, we have that $\zeta(s)^z=\sum_{n=1}^{\infty}d_z(n)/n^s$ for Re$(s)>1$, where $d_z(n)$
is  the ``$z$-th divisor function'', defined as the  multiplicative function such that $d_z(p^a)=\Gamma(z+a)/\Gamma(z)a!$,
for any prime $p$ and any integer $a\geq 0$.  Our knowledge of the $2k$-moments of $\zeta(\sigma+it)$  for $1/2<\sigma<1$
is very incomplete, and we only have asymptotic formulas in a certain restricted range of $k$. Indeed we know that for
any $\sigma>1/2$ there is a real number $\kappa(\sigma)$ such that for any positive integer $k\leq \kappa(\sigma)$ we have that
\begin{equation}
\frac{1}{T}\int_{T}^{2T}|\zeta(\sigma+it)|^{2k}dt= (1+o(1))\sum_{n=1}^{\infty}\frac{d_k^2(n)}{n^{2\sigma}}.
\end{equation}
In fact we know that $\kappa(\sigma)\geq 1/(1-\sigma)$ for $1/2<\sigma<1$ (see Theorem 7.7 of [27]) and that
$\kappa(\sigma)=\infty$ for $\sigma\geq 1$. Moreover, it is conjectured that $\kappa(\sigma)=\infty$ for all $\sigma>1/2$.
 This assumption is equivalent to the Lindel\"of hypothesis for $\zeta(s)$ (see Theorem 13.2 of [27]).
On the line $\sigma=1$, Granville and Soundararajan (see Theorem 2 of [14]) proved unconditionally that
(1.7) holds uniformly in the range $k\ll (\log T)/(\log_2 T)^2$ (which we have slightly improved to
$k\ll \log T/(\log_2T\log_3 T)$ in [16]), and an analogous argument to Theorem 1.3 below, implies
that the asymptotic formula (1.7) does not hold for $k\geq C\log T\log_2T$, if $C$ is suitably large.
For $1/2<\sigma<1$ no uniform version of (1.7) is known  even on the Lindel\"of hypothesis, and one wonders
if a stronger assumption, namely the RH, would imply (1.7) uniformly in some range $k\leq F_{\sigma}(T)$ where
 $F_{\sigma}(T)\to \infty$ at $T\to\infty$. The answer is definitely yes and even more! In fact assuming RH
we can also   handle complex moments, allow $\sigma$ to be close to $1/2$ and get an explicit error term in (1.7).

\begin{thm} Assume the Riemann hypothesis. Then there exist positive constants $K$ and $b(K)$,
such that uniformly for $1/2+K/\log_2T\leq \sigma\leq 1-\epsilon $ we have
$$ \frac{1}{T}\int_{T}^{2T}\zeta(\sigma+it)^{z_1}\overline{\zeta(\sigma+it)}^{z_2}dt= \sum_{n=1}^{\infty}
\frac{d_{z_1}(n)d_{z_2}(n)}{n^{2\sigma}}+O\left(\exp\left(-\frac{\log T}{50\log_2 T}\right)\right),$$ for
all complex numbers $z_1, z_2$ with $|z_i|\leq b(K)(\log T)^{2\sigma-1}$.
\end{thm}
As a consequence of this result we know that for any $1/2<\sigma<1$ the asymptotic formula (1.7) holds for all
 integers $k\ll (\log T)^{2\sigma-1}$ assuming RH, and one wonders if it still holds for even bigger values of $k$.
 First, using an idea of Farmer, Gonek and Hughes [5], we prove in Theorem 1.3 below that (1.7) does not hold in the
range $k\geq (B(\sigma)+\epsilon)(\log T\log_2T)^{\sigma}$, for a certain positive constant $B(\sigma)$. Moreover,
 given $0<\delta\leq \sigma$, we show in Theorem 1.4 that the validity of the asymptotic formula (1.7) in the range
$k\ll (\log T)^{\delta}$ is essentially equivalent to the fact that $\max_{t\in [T,2T]}\log |\zeta(\sigma+it)|\ll (\log T)^{1-\delta}(\log_2 T)^{O(1)}$.
 Finally in Theorem 1.5, we use a recent method of Rudnick and Soundararajan [21], to show that the lower bound for the moments in (1.7),
 holds in the range $k\ll(\log T)^{\sigma}$. We should note that these results are unconditional. For $T$ large and $1/2<\sigma<1$,
define $L_T(\sigma):=\max_{t\in [T,2T]}\log |\zeta(\sigma+it)|$, and $G_1(\sigma):=\int_{0}^{\infty}\log I_0(u)u^{-1-\frac{1}{\sigma}}du.$
\begin{thm} Let $\epsilon>0$ be small. Then the asymptotic formula (1.7) does not hold for any real number
 $k$ in the range
$$ k\geq \frac12\left((B(\sigma)+\epsilon)\log T\log_2 T\right)^{\sigma},$$ where
$B(\sigma):=\sigma^{(1-2\sigma)/(1-\sigma)}/(G_1(\sigma)(1-\sigma)).$
\end{thm}

\begin{thm} Let $1/2<\sigma<1$ and $0<\delta\leq \sigma$. If (1.7) holds for all positive integers $k\leq (\log T)^{\delta}$
then $L_T(\sigma)\ll (\log T)^{1-\delta}.$ Conversely if $L_T(\beta)\ll (\log T)^{1-\delta}$, for all $\beta \geq \sigma-1/\log_2 T$,
then the asymptotic formula (1.7) holds for all positive integers $k\leq c(\log T)^{\delta}/\log_2 T$ where $c>0$ is a suitably small constant.
\end{thm}

\begin{thm} Let $0<\alpha<1$ be a real number. Then there exists $c>0$ such that uniformly for any
$1/2+1/(\log T)^{\alpha}<\sigma<1-\epsilon$ and all positive integers $k\leq c((2\sigma-1)\log T)^{\sigma}$, we have
$$\frac{1}{T}\int_{T}^{2T}|\zeta(\sigma+it)|^{2k}dt\geq \sum_{n=1}^{\infty}\frac{d_k^2(n)}{n^{2\sigma}}
\left(1+O\left(\exp\left(-\frac{\log T}{10k\log_2 T}\right)\right)\right).$$
\end{thm}

{\it Remark 2.}  If (1.7) holds in the range  $k\leq \frac12\left((B(\sigma)+o(1))\log T\log_2 T\right)^{\sigma}$
then the proof of Theorem 1.3 gives that $L_T(\sigma)= (C(\sigma)+o(1))(\log T)^{1-\sigma}(\log_2 T)^{-\sigma}$,
where $C(\sigma):=G_1(\sigma)^{\sigma}\sigma^{-2\sigma}(1-\sigma)^{\sigma-1}$. Moreover, if this is the case then
the lower bound in Theorem 1.5 does not hold in the range $k\geq c(\log T\log_2T)^{\sigma}$ for any $c>\frac12(B(\sigma))^{\sigma}$.

Concerning other families of $L$-functions, P.D.T.A Elliott [4] has established the analogue of Bohr and Jessen's
result for the family of quadratic Dirichlet $L$-functions, at a fixed point $s$, with $1/2<\text{Re}(s)\leq 1$.
Furthermore he showed that the limiting distribution function for these values is smooth, and obtained a formula
for its characteristic function. In [8], Granville and Soundararajan studied the distribution of extreme values of
this family at $s=1$ and proved that the tail of the distribution has a similar asymptotic to $\Phi_T(1,\tau)$ (see (1.3))
but with a different constant $a_1$. Inside the critical strip, our method can be generalized to provide estimates for
the distribution of large values of families of $L$-functions, at a fixed point  $1/2<\sigma<1$ (analogous results are
also proved for the distribution of small values). As a first example we show that the corresponding result for the values
 $\log|L(\sigma,\chi)|$ (and $\arg L(\sigma,\chi)$) as $\chi$ varies over non-principal characters modulo a large prime
$q$, holds almost verbatim, just changing $T$ to $q$ in Theorem 1.1 (see Theorem 4.5 in section 4).

 Furthermore, let $\Phi_{x}^{\text{quad}}(\sigma,\tau)$ be the proportion of fundamental discriminants $d$ such that
 $|d|\leq x$ and  $\log L(\sigma,\chi_d)>\tau.$ That is
$$ \Phi_{x}^{\text{quad}}(\sigma,\tau):=\left(\sums_{|d|\leq x}1\right)^{-1}
\sums_{\substack{|d|\leq x\\ \log |L(\sigma,\chi_d)|>\tau}} 1,$$
where $\sums$ indicates that the sum is over fundamental discriminants. Exploiting ideas of Granville and Soundararajan [8]
  and appealing to a remarkable result of Graham and Ringrose [6] on bounds for character sums to smooth moduli,
we increase the range of uniformity where $\Phi_{x}^{\text{quad}}(\sigma,\tau)$ can be estimated from a range
$\tau \ll (\log x)^{1-\sigma}/\log_2 x$ (the analogue of Theorem 1.1) to a range $\tau\ll(\log x\log_4x)^{1-\sigma}/\log_2 x$.
 We should note that this improvement is of some interest since we believe that the maximum of the values $\log L(\sigma,\chi_d)$
 over fundamental discriminants $d$ with $|d|\leq x$ is $\asymp (\log x)^{1-\sigma}/(\log_2 x)^{\sigma}.$

\begin{thm}Let $1/2<\sigma<1$ and $x$ be large. Then there exists $c_2(\sigma)>0$ such that uniformly
in the range $1\ll \tau\leq c_2(\sigma)(\log x\log_4x)^{1-\sigma}/\log_2 x$ we have
$$
\Phi_{x}^{\text{quad}}(\sigma,\tau)=\exp\left(-A_2(\sigma)\tau^{\frac{1}{(1-\sigma)}}(\log\tau)^{\frac{\sigma}{(1-\sigma)}}
\left(1+O\left(\frac{1}{\sqrt{\log\tau}}+r(y,\tau)\right)\right)\right),
$$
where $y= \log x\sqrt{\log_3 x}$, and $A_2(\sigma)=A_X(\sigma)$ (see (1.5)) with $X$ being a random variable taking the
values $1$ and $-1$ with equal probability $1/2$. In this case we should note that  $\ex(e^{tX})=\cosh(t)$.
\end{thm}

Let $q$ be a large prime and denote by $S_2^p(q)$ the set of arithmetically normalized primitive holomorphic cusp
forms of weight $2$ and level $q$. Then every $f\in S_2^p(q)$ has a Fourier expansion
$f(z)=\sum_{n=1}^{\infty} \lambda_f(n)\sqrt{n}e^{2\pi in z}, \text{ for } \text{Im}(z)>0.$ The $L$-function attached
 to $f$ is defined  for Re$(s)>1$ by $L(s,f)=\sum_{n=1}^{\infty}\lambda_f(n)n^{-s}$.
In [2],  Cogdell and Michel obtained asymptotic formulas for complex moments of this family at $s=1$;
and Liu, Royer and Wu [17] proved that the tail of the distribution of the values $\log L(1,f)$ has the same shape as (1.3).
Combining our method with a zero density result of Kowalski and Michel [13], we get the analogue of
 Theorem 1.1 for this family. In view of the Petersson trace formula, it is arguably more natural to consider
the weighted arithmetic distribution function
$$ \Phi_{q}^{\text{aut}}(\sigma,\tau):=\left(\sum_{f\in S_2^p(q)}\omega_f\right)^{-1}
\sum_{\substack{f\in S_2^p(q)\\ \log L(\sigma,f)>\tau}} \omega_f,$$
where $ \omega_f:=1/(4\pi \langle f,f\rangle)$ is the usual harmonic weight, and $\langle f,g\rangle$ is
the Petersson inner product on the space $\Gamma_0(q)\backslash\Bbb{H}$. We prove
\begin{thm}
Let $1/2<\sigma<1$, and $q$ be a large prime. Then there exists $c_3(\sigma)>0$ such that uniformly in the
 range $1\ll \tau\leq c_3(\sigma)(\log q)^{1-\sigma}/\log_2 q$ we have
$$ \Phi_{q}^{\text{aut}}(\sigma,\tau)=\exp\left(-A_3(\sigma)\tau^{\frac{1}{1-\sigma}}(\log\tau)^{\frac{\sigma}{1-\sigma}}\left(1+O\left(\frac{1}{
\sqrt{\log\tau}}+ r(\log q,\tau)\right)\right)\right),
$$
where $A_3(\sigma)=A_X(\sigma)$ (see (1.5)) with $X=2\cos\theta$ and $\theta$ being a random variable distributed
on $[0,\pi]$ according to the Sato-Tate measure $\frac{2}{\pi}\sin^2t dt$.
\end{thm}

As a corollary of Theorems 4, 5 and 4.5 we can produce large values of $L$-functions when averaged over families.
 Indeed we can show that the logarithm of the absolute value of the $L$-function at $1/2<\sigma<1$ in the corresponding
 family, can be as large as   $(\log Q)^{1-\sigma}/\log_2 Q,$ where $Q$ is the conductor of the family.
This can also be derived by a ``resonance'' method of Soundararajan [25] which produces large values of $L$-functions
on the critical line. However, the analogue of  Montgomery's $\Omega$ result (1.2) is not known to hold for other families
of $L$-functions, since his method does not appear to generalize to this
situation; and it is certainly interesting to prove such a result in an other context than for $\zeta(\sigma+it)$.
We achieved this, conditionally on the Generalized Riemann Hypothesis GRH, for Dirichlet $L$-functions attached
to quadratic characters of prime moduli.  Let $\chi_p=\left(\frac{\cdot}{p}\right)$ denote the Legendre symbol
modulo a prime $p$.  In [18], Montgomery established that if the GRH is true  then there are infinitely many
primes $p$ such that the least quadratic non-residue $(\text{mod } p)$ is $\gg \log p\log_2 p$. This idea has
been exploited by Granville and Soundararajan [8] to examine extreme values of $L(1,\chi)$ on GRH. We adapt
this technique to our setting and show that
\begin{thm} Assume GRH. Let $s=\sigma+it$ where $1/2< \sigma<1$, and $t\in \Bbb{R}$. Let $x$ be large.
Then there are $\gg x^{1/2}$ primes $p\leq x$ such that
$$ \log |L(s,\chi_p)|\geq (\beta(s)+o(1))(\log x)^{1-\sigma}(\log_2 x)^{-\sigma},$$
and $\gg x^{1/2}$ primes $q\leq x$ such that
$$ \log |L(s,\chi_q)|\leq -(\beta(s)+o(1))(\log x)^{1-\sigma}(\log_2 x)^{-\sigma},$$
where $\beta(\sigma)=(2\log 2)^{\sigma-1}/(1-\sigma)$ and $\beta(s)=\beta(\sigma)t^2/\left((1-\sigma)^2+4 t^2\right)$ if $t\neq 0$.
\end{thm}

{\it Remark 3.} When $t=0$, notice that $\beta(\sigma)>\sqrt{\frac{2}{\log 2}}\approx 1.698$ is larger than $1/20$
which corresponds to Montgomery's $\Omega$ result for  $\zeta(\sigma+it)$ under the assumption of RH (see (1.2)).

We now describe the probabilistic part of our work. Let $\mathcal{L}=\{L(s,\pi), \pi \in \mathcal{F}\}$ be a family of $L$-functions
attached to a set of arithmetic objects $\mathcal{F}$ (characters, modular forms, ...), where $L(s,\pi)$ have degree $d$ for some
$d\in \Bbb{N}$, that is $L(s,\pi)=\prod_{p}\prod_{j=1}^d (1-\alpha_{j,\pi}(p)p^{-s})^{-1}$ for Re$(s)>1$. Then one expects that
 as $\pi$ varies in $\mathcal{F}$ and $|\mathcal{F}|\to\infty$, the local roots $\alpha_{j,\pi}(p)$ should behave like random variables
  $X_j(p)$ which are expected to be independent for different primes (at least for small primes). Then we model the values of
 $L(s,\pi)$ by the random Euler product $L(s,\X)=\prod_{p}\prod_{j=1}^d (1-X_j(p)p^{-s})^{-1}$, which is absolutely
convergent a.s. for Re$(s)>1/2$, provided that the $X_j(p)$ are bounded and that $\ex(X_j(p))=0$.

Instead of studying the probabilistic random model for each family, we construct a class of these models which satisfy
some natural conditions, and can be useful to model even more general families of $L$-functions (for example symmetric powers $L$-functions of holomorphic forms).
Let $d$ be a positive integer, and let $\X(p)=(X_1(p), X_2(p),..., X_d(p))$
be independent random vectors of dimension $d$ indexed by the prime numbers $p$, where the $X_j(p)$ are random variables
 defined on a probability space $(\Omega,\mu)$, and taking values on a disk $\Bbb{D}(M)=\{z\in \Bbb{C}, |z|\leq M\}$ of
the complex plane, where $M$ is some absolute constant. Moreover we assume that the mean vector
${\Bbb {\bold{E}}}({\X}(p)):=(\ex(X_1(p)), \ex(X_2(p)),..., \ex(X_d(p)))=(0,0,...,0)$ for all primes $p$.
For a real number $y\geq 2$ we define the following random product

$$ L(s,\X;y):
=\prod_{p\leq y}\prod_{j=1}^d\left(1-\frac{X_j(p)}{p^s}\right)^{-1}.$$
Our aim is to study the distribution of the random variables $\log|L(\sigma,\X;y)|$, and
$\arg L(\sigma,\X;y)$ for $1/2<\sigma<1$. Specifically we intend to estimate the tails of distribution
$$\Phi_X(\tau;y):=\pr\left(\log|L(\sigma,\X;y)|>\tau\right), \text{ and } \Psi_X(\tau;y):=\pr\left(\arg L(\sigma,\X;y)>\tau\right)$$
uniformly in $y$ and $\tau$ (letting $y\to\infty$ we also get information on the distribution of       $L(s,\X)$).
In fact we shall see that the distribution of $\log |L(\sigma,\X;y)|$ (respectively that of $\arg L(\sigma,\X;y)$)
is governed by the distribution of the random variables
$ Z(p):=\sum_{j=1}^d\text{Re}X_j(p)$ (respectively $Y(p):=\sum_{j=1}^d\text{Im}X_j(p)$). In fact the only condition
we need in order to estimate the Laplace transform of $\log|L(\sigma,\X;y)|$ (respectively $\arg L(\sigma,\X;y)$) is
 that the sequence $\{Z(p)\}_p$  (respectively $\{Y(p)\}_p$) converges in distribution to some random variable $Z$ (respectively $Y$), in a uniform way:

{\bf Uniform limiting distribution hypothesis (ULD)}.
We say that a sequence of random variables $\{X(p)\}_{p \text{ prime }}$ satisfies (ULD) if there exists a
 random variable $X$ such that for any $A>0$, and large primes $p$ we have
$$ \ex\left(e^{tX(p)}\right)=\ex\left(e^{tX}\right)\left(1+O_A\left(\frac{1}{\log^A p}\right)\right),  \text{  uniformly for  all } t\in \Bbb{R}.$$
\begin{thm} Let $\tau$ be large and $y\geq (\tau\log\tau)^{1/(1-\sigma)}$ be a real number.
Assume that the sequence $\{Z(p)\}_{p}$ satisfies hypothesis (ULD), and denote by $Z$ the random
 variable to which it converges in distribution. Then we have
\begin{equation}
 \Phi_X(\tau;y)=\exp\left(-A_Z(\sigma)\tau^{\frac{1}{1-\sigma}}(\log\tau)^{\frac{\sigma}{1-\sigma}}
\left(1+O\left(\frac{1}{\sqrt{\log \tau}}+r(y,\tau)\right)\right)\right),
\end{equation}
where $A_Z(\sigma)$ and $r(y,\tau)$ are defined by (1.5) and (1.6) respectively.
Furthermore, if the sequence $\{Y(p)\}_{p}$ satisfies hypothesis (ULD), and $Y$ is the random variable
to which it converges in distribution, then $\Psi_X(\tau;y)$ has the asymptotic (1.8) with $A_Z(\sigma)$ replaced by $A_Y(\sigma).$
\end{thm}
{\it Remark 4.} If $Z$ is symmetric (that is $Z$ and $-Z$ are identically distributed) then we obtain
the same asymptotic for  $\pr\left(\log|L(\sigma,\X;y)|<-\tau\right)$.

{\it Remark 5.} This theorem is an improvement of a recent work of Hatori and Matsumoto [10], who found
an asymptotic formula for  $\log \pr(\sum_pZ(p)/p^{\sigma}>\tau)$ (without an explicit error term), where
$Z(p)$ are bounded real valued identically distributed random variables with $\ex(Z(p))=0$. Their method
relies on a Tauberian theorem of exponential type. We should also note that their approach is more general
and provides asymptotics for the distribution of $\sum_n Z(n)r_n$, where $\{r_n\}$ is  a regularly varying
sequence of index $-\sigma$. However, in the special case where $r_n=p^{-\sigma}$, our method is simpler,
more effective and does not use these Tauberian type arguments.

\section{Preliminaries}

\subsection{Estimates for divisor functions}

Here and throughout $S(y)$ denotes the set of $y$-smooth numbers, defined to be positive integers $n$ whose prime factors are below $y$.
In this section we collect some useful estimates  for the divisor function $d_z(n)$. First we recall some easy bounds borrowed from [8]. We have that
$ |d_z(n)|\leq d_{|z|}(n)\leq d_k(n),$
for any integer $k\geq |z|.$ If $a$ and $b$ are positive integers then $d_a(n)d_b(n)\leq d_{ab}(n)$ for all $n\in \Bbb{N}$. We also record that $d_a(n^2)d(n)\leq d_{2a+2}(n)^2$.  These inequalities may be shown by first proving them for prime powers, and then using multiplicativity.
Let $k$ be a positive integer. Then for $1/2<\sigma<1$ we have that
\begin{equation}
 \begin{aligned}
\sum_{n\in S(y)}\frac{d_k(n)}{n^{\sigma}}&=\prod_{p\leq y}\left(1-\frac{1}{p^{\sigma}}\right)^{-k}=\exp\left(k\sum_{p\leq y}\frac{1}{p^{\sigma}}+O(k)\right)\\
&=\exp\left(\frac{ky^{1-\sigma}}{(1-\sigma)\log y}+O\left(\frac{ky^{1-\sigma}}{(1-\sigma)^2\log^2 y}\right)\right),
\end{aligned}
\end{equation}
using the prime number theorem. Let  $X>3$ be a real number. Then
 $d_k(n)e^{-n/X}\leq
e^{k/X}\sum_{a_1...a_k=n}e^{-(a_1+...+a_k)/X}$, which implies that
\begin{equation}
 \sum_{n=1}^{\infty}\frac{d_k(n)}{n}e^{-n/X}\leq \left(e^{1/X}
\sum_{a=1}^{\infty}\frac{e^{-a/X}}{a}\right)^k\leq (\log
3X)^k.
\end{equation}

Furthermore we note that for any  $0\leq \sigma\leq 1$ we have that
\begin{equation}
 \begin{aligned}
\sum_{n=1}^{\infty}\frac{d_k(n)}{n^{\sigma}}e^{-n/X}&=\sum_{n\leq 3X\log X}+\sum_{n>3X\log X}\frac{d_k(n)}{n^{\sigma}}e^{-n/X}\\
&\leq (3X\log X)^{1-\sigma}\sum_{n=1}^{\infty}\frac{d_k(n)}{n} e^{-n/X}+\sum_{m>3X\log X}\frac{d_k(m)}{m}e^{-m/(2X)}\\
& \ll (X\log X)^{1-\sigma}\left(\log 3X\right)^{k}.\\
\end{aligned}
\end{equation}
Let $z_1$ and $z_2$ be complex numbers. Then for any $\sigma>1/2$ we have
\begin{equation}
 \sum_{n=1}^{\infty}
\frac{d_{z_1}(n)d_{z_2}(n)}{n^{2\sigma}}=\prod_{p}\left(\frac{1}{2\pi}\int_{-\pi}^{\pi}\left(1-\frac{e^{i\theta}}{p^{\sigma}}\right)^{-z_1}
\left(1-\frac{e^{-i\theta}}{p^{\sigma}}\right)^{-z_2}d\theta\right).
\end{equation}
This follows by multiplicativity upon noting that
\begin{equation*}
\begin{aligned}
&\frac{1}{2\pi}\int_{-\pi}^{\pi}\left(1-\frac{e^{i\theta}}{p^{\sigma}}\right)^{-z_1}
\left(1-\frac{e^{-i\theta}}{p^{\sigma}}\right)^{-z_2}d\theta=\frac{1}{2\pi}\int_{-\pi}^{\pi}\sum_{a=0}^{\infty}
\frac{d_{z_1}(p^a)e^{i\theta a}}{p^{\sigma a}}\sum_{b=0}^{\infty}\frac{d_{z_2}(p^b)e^{-i\theta b}}{p^{\sigma b}}d\theta\\
&= \sum_{a,b\geq 0}
\frac{d_{z_1}(p^a)d_{z_2}(p^b)}{p^{\sigma(a+b)}}\frac{1}{2\pi}\int_{-\pi}^{\pi}e^{i(a-b)\theta}d\theta=\sum_{a=0}^{\infty}
\frac{d_{z_1}(p^a)d_{z_2}(p^a)}{p^{2\sigma a}}.
\end{aligned}
\end{equation*}
Finally we prove
\begin{lem} There exists $C>0$ such that for any $1/2<\sigma<1$, and $k>0$ large, we have
$$\sum_{n=1}^{\infty}\frac{d_k(n)^2}{n^{2\sigma}}\leq \exp\left(C\frac{k^{1/\sigma}}{(2\sigma-1)(1-\sigma)\log k}\right).$$
\end{lem}

\begin{proof}[Proof] For a prime $p$, let $E_p(k)=\frac{1}{2\pi}\int_{-\pi}^{\pi}\left(1-e^{i\theta}/p^{\sigma}\right)^{-k}
\left(1-e^{-i\theta}/p^{\sigma}\right)^{-k}d\theta$. Then for primes $p\leq (2k)^{1/\sigma}$ we use that $E_p(k)\leq
(1-1/p^{\sigma})^{-2k}$, and for primes $p>(2k)^{1/\sigma}$ we have that $ E_p(k)=I_0(2k/p^{\sigma})(1+O(k/p^{2\sigma}))$.
 Now using the prime number theorem,  equation (2.4) with $z_1=z_2=k$, along with the fact that $\log I_0(t)=O(t^2)$
for $0<t\leq 1$, we deduce that
$$ \log \left(\sum_{n=1}^{\infty}\frac{d_k(n)^2}{n^{2\sigma}}\right)\ll \sum_{p\leq (2k)^{1/\sigma}}\frac{k}{p^{\sigma}}+
 \sum_{p\geq (2k)^{1/\sigma}}\frac{k^2}{p^{2\sigma}}\ll \frac{k^{1/\sigma}}{\log k}\left(\frac{\sigma}{1-\sigma}+\frac{\sigma}{2\sigma-1}\right).$$
\end{proof}
\subsection{ Approximating $L$-functions by short Euler products}

 We begin by stating the following approximation lemmas which have been proved in [9] and [7] for the Riemann zeta function and Dirichlet $L$-functions, and in [2] for $L$-functions attached to holomorphic cusp forms of weight 2 and large level.  These results will later be combined with zero-density estimates, to show that with very few exceptions, the $L$-functions belonging to one of the families we are considering can be approximated by very short Euler products (over the primes $p\leq (\log Q)^A$, where $Q$ is the conductor of the corresponding family) in the strip $1/2<\text{Re}(s) <1$. We have
\begin{lem}[{Lemma 1 of [9]}]
Let $z\ge 2$ and $|t| \ge z+3$ be real numbers.
Let $\frac{1}{2} \leq \sigma_0 <\sigma \leq 1$ and suppose that the
rectangle $\{ s: \ \ \sigma_0 <\text{Re}(s) \leq 1, \ \
|\text{Im}(s) -t| \leq z+2\}$ does not contain zeros of $\zeta(s)$.
Then
$$
\log \zeta(\sigma+it)= \sum_{n=2}^{z} \frac{\Lambda(n)}{n^{\sigma+it}\log n} +
O\left( \frac{\log
|t|}{(\sigma_1-\sigma_0)^2}z^{\sigma_1-\sigma}\right),
$$
where  $\sigma_1 = \min(\sigma_0+\frac{1}{\log z},
\frac{\sigma+\sigma_0}{2})$.
\end{lem}
\begin{lem}[{Lemma 8.2 of [7]}] Let $q$ be  a large prime and $\chi$ a character $(\text{mod } q)$. Let  $z\geq 2$ and $|t|\leq 3q$ be real numbers.
Let $\frac{1}{2} \leq \sigma_0 <\sigma\leq 1$ and suppose that the
rectangle $\{ s: \ \ \sigma_0 <\text{Re}(s) \leq 1, \ \
|\text{Im}(s) -t| \leq z+2\}$ does not contain zeros of $L(s,\chi)$.
Then
$$
\log L(\sigma+it,\chi)= \sum_{n=2}^{z} \frac{\Lambda(n)\chi(n)}{n^{\sigma+it}\log n} +
O\left( \frac{\log
q}{(\sigma_1-\sigma_0)^2}z^{\sigma_1-\sigma}\right),
$$
where  $\sigma_1 = \min(\sigma_0+\frac{1}{\log z},
\frac{\sigma+\sigma_0}{2})$.
\end{lem}

Let $q$ be a large prime and $f\in S_2^p(q)$. Then Deligne's Theorem implies that for all primes $p\neq q$ there exists $\theta_f(p)\in [0,\pi]$ such that $\lambda_f(p)=2\cos\theta_f(p)$. We have
\begin{lem} [{Lemma 4.3 of [2]}] Let $2\leq z<q $ and $|t|\leq 2q$ be real numbers.
Let $\frac{1}{2} \leq \sigma_0 <\sigma\leq 1$ and suppose that the
rectangle $\{ s: \ \ \sigma_0 <\text{Re}(s) \leq 1, \ \
|\text{Im}(s) -t| \leq z+2\}$ does not contain zeros of $L(s,f)$.
Then
$$
\log L(\sigma+it,f)= \sum_{n=2}^{z} \frac{\Lambda(n)b_f(n)}{n^{\sigma+it}\log n} +
O\left( \frac{\log
q}{(\sigma_1-\sigma_0)^2}z^{\sigma_1-\sigma}\right),
$$
where $\sigma_1 = \min(\sigma_0+\frac{1}{\log z},
\frac{\sigma+\sigma_0}{2})$ and $b_f(n)=(e^{i\theta_f(p)})^m+(e^{-i\theta_f(p)})^m$ if $n=p^m$ for some prime $p$, and equals $0$ otherwise.
\end{lem}

In some cases it is helpful to approximate short Euler products by Dirichlet polynomials. Our next lemma shows that this is possible if the coefficients are bounded by some divisor function. This will be used in order to apply the Petersson trace formula to compute moments of short Euler products of automorphic $L$-functions (see section 6 below).
\begin{lem}
Let $g(n)$ be a multiplicative function such that $g(n)\ll d_k(n)$ for some positive integer $k$. Let $y>2$ be a real number and define
$$ L(s,g;y):=\sum_{n\in S(y)}\frac{g(n)}{n^s}, \text{ for } s\in \Bbb{C}.$$
Then for $1/2<\text{Re(s)}=\sigma<1$ and $x\geq y^2$ we have
$$ L(s,g;y)=\sum_{\substack{ n\leq x\\n\in S(y)}} \frac{g(n)}{n^s}+ O\left(\exp\left(-\frac{\log x}{\log y}+
 \log_2 x+ \frac{eky^{1-\sigma}}{(1-\sigma)\log y}(1+o(1))\right)\right).$$
\end{lem}

\begin{proof}[Proof]
Without loss of generality we may assume that $x\in \Bbb{Z}+1/2$. We use Perron's formula (See [3]). Let $c=1/\log x$ and $T=x^2$, then we have
\begin{equation}\frac{1}{2\pi i} \int_{c-iT}^{c+iT} L(s+z,g;y)\frac{x^z}{z}dz=\sum_{\substack{ n\leq x\\n\in S(y)}} \frac{g(n)}{n^s}+ E_1,\end{equation}
where
$$ E_1\ll \frac1T \sum_{n\in S(y)}\frac{d_k(n)x^c}{n^{\sigma+c}|\log(x/n)|}\ll \frac{x}{T}\sum_{n\in S(y)}\frac{d_k(n)}{n^{\sigma+c}}\ll \exp\left(-\log x+O\left(\frac{ky^{1-\sigma}}{\log y}\right)\right),$$
using (2.1) along with the fact that $\log(x/n)\gg 1/x$. Now we move the line of integration to the line $\text{Re}(s)=-\beta$ where $\beta=1/\log y$. We encounter a simple pole at $s=0$ which leaves the residue $L(s,g;y)$. It follows from (2.1) that the LHS of (2.5) equals $L(s,g;y)$ plus
\begin{equation*}
\begin{aligned}
&\frac{1}{2\pi i}\left(\int_{c-iT}^{-\beta-iT}+\int_{-\beta-iT}^{-\beta+iT}
+\int_{-\beta+iT}^{c+iT}\right)L(s+z,g;y)\frac{x^z}{z}dz\\
&\ll \frac{1}{T}\exp\left(O\left(\frac{ky^{1-\sigma}}{\log y}\right)\right)+x^{-\beta}\log T\exp\left(\frac{eky^{1-\sigma}}{(1-\sigma)\log y}(1+o(1))\right)\\
&\ll \exp\left(-\frac{\log x}{\log y}+ \log_2 x+ \frac{eky^{1-\sigma}}{(1-\sigma)\log y}(1+o(1))\right).
\end{aligned}
\end{equation*}
\end{proof}

\section{Random Euler products and their distribution}
 For a random variable $Y$, the cumulant-generating function of $Y$ if it exists, is defined by
$ g_Y(t):=\log \ex\left(e^{tY}\right)=\sum_{n=1}^{\infty}\kappa_nt^{n}/n!,$
 where $\kappa_n$ are the cumulants of $Y$. Moreover one has $\kappa_1=\ex(Y)$ and $\kappa_2=\text{Var}(Y)$.
 Our first lemma describes some useful properties and estimates for  the function $g_Y$.
\begin{lem} Let $Y$ be a bounded real-valued random variable such that $\ex(Y)=0$. Then
$g_Y$ is smooth on $\Bbb{R}$ (is of class $C^{\infty}$) and $|g_Y'(t)|\ll 1$. Moreover, we have that $g_Y(t)=O(t^2)$ if $0\leq t\leq 1$ and $g_Y(t)=O(t)$ if $t\geq 1$.
\end{lem}

\begin{proof}
Let $M_Y(t)=\ex(e^{tY})$ be the moment-generating function of $Y$. Since $Y$ is bounded
then $M_Y$ is a positive smooth function which implies that $g_Y$ is smooth.
Moreover, it is easy to check that $M_Y'(t)=\ex(Ye^{tY})$, simply by differentiating the Taylor series expansion of $M_Y$.
 Then the first assumption follows upon noting that  $g_Y'(t)=M_Y'(t)/M_Y(t)$ and $|M_Y'(t)|=|\ex(Ye^{tY})|\leq \ex(|Y|e^{tY})\ll M_Y(t).$

 The estimate for $g_Y$ on  $[0,1]$, follows from its Taylor expansion along with the fact that $\ex(Y)=0$.
Now for $t\geq 1$, this follows from the facts that $Y$ is bounded and that  $g_Y(t)=\log\ex(e^{tY})$.
\end{proof}
It follows from this lemma that if $Y$ is a bounded real-valued random variable with mean $0$, then $G_Y(\sigma)=\int_{0}^{\infty}\log \ex(e^{uY})u^{-1-\frac{1}{\sigma}}du$ is absolutely convergent for any $1/2<\sigma<1$.
In order to prove Theorem 1.9 we shall compute large moments of the random variable $L(\sigma,\X;y)$.

\begin{pro} Assume that the sequence $\{Z(p)\}_{p}$ satisfies hypothesis (ULD), and denote by $Z$ the random variable to which it converges in distribution. Let $r$ be large and $y\geq r^{1/\sigma}$ be a real number. Then we have
$$ \log\ex\left(|L(\sigma,\X;y)|^{r}\right)=G_Z(\sigma)\frac{r^{1/\sigma}}{\log r}\left(1+O\left(\frac{1}{\log r}+\left(\frac{r^{1/\sigma}}{y}\right)^{2\sigma-1}\right)\right).
$$
If $Z$ is symmetric then we get the same estimate for $\log\ex\left(|L(\sigma,\X;y)|^{-r}\right)$.
\end{pro}

\begin{pro} Assume that the sequence $\{Y(p)\}_{p}$ satisfies hypothesis (ULD), and denote by $Y$ the random variable to which it converges in distribution. Let $r$ be large and $y\geq r^{1/\sigma}$ be a real number. Then we have
$$ \log\ex\left(L(\sigma,\X;y)^{-i r}\overline{L(\sigma,\X;y)}^{ir}\right)=G_Y(\sigma)\frac{r^{1/\sigma}}{\log r}\left(1+O\left(\frac{1}{\log r}+\left(\frac{r^{1/\sigma}}{y}\right)^{2\sigma-1}\right)\right).$$
If $Y$ is symmetric then we get the same estimate for
 $\log\ex\left(L(\sigma,\X;y)^{i r}\overline{L(\sigma,\X;y)}^{-ir}\right)$.
\end{pro}

Using equation (2.4) one can observe that $\ex\left(|\zeta(\sigma,\X)|^{2k}\right)=\sum_{n=1}^{\infty}d_k(n)^2/n^{2\sigma}.$ Then from  Proposition 3.2 we can deduce the following corollary
\begin{cor}
Let $1/2<\sigma<1$, and $k$ be a large positive real number.
Then
$$\sum_{n=1}^{\infty}\frac{d_k(n)^2}{n^{2\sigma}}=\exp\left(G_1(\sigma)\frac{(2k)^{1/\sigma}}{\log k}\left(1+O\left(\frac{1}{\log k}\right)\right)\right).
$$
\end{cor}

\begin{proof}[Proof of Proposition 3.2] For a prime number $p$, let $f_p(t):=\log\ex(e^{tZ(p)})$ be the cumulant-generating function of $Z(p)$. Define $$E_p(r):=\ex\left(\prod_{j=1}^d
\left|1-\frac{X_j(p)}{p^{\sigma}}\right|^{-r}\right)=\ex\left(\prod_{j=1}^d
\left(1-\frac{2\text{Re}X_j(p)}{p^{\sigma}}+\frac
{1}{p^{2\sigma}}\right)^{-r/2}\right).$$
 Then by the independence of the $\X(p)$ we know that $\ex(|L(\sigma,\X;y)|^r)=\prod_{p\leq y}E_p(r)$.
 If $p\leq r^{1/(2\sigma)}$, then $\log E_p(r)=O(r/p^{\sigma})$, which follows simply from the fact that the $X_j(p)$ are bounded. Now for primes $p$ such that $r^{1/{2\sigma}}<p\leq y
$, we have that
$$
\prod_{j=1}^d
\left(1-\frac{2\text{Re}X_j(p)}{p^{\sigma}}+\frac
{1}{p^{2\sigma}}\right)^{-r/2}
=\exp\left(\frac{r}{p^{\sigma}}Z(p)\right)\left(1+
O\left(\frac{r}{p^{2\sigma}}\right)\right),
$$
so that $\log E_p(r)=f_p\left(r/p^{\sigma}\right)+O\left(r/p^{2\sigma}\right)$.
Hence combining these estimates we deduce that
$$
\log\ex\left(|L(\sigma,\X;y)|^{r}\right)=
 \sum_{r^{1/(2\sigma)}<p\leq y}f_p\left(\frac{r}{p^{\sigma}}\right)+ E_2,
$$
where
$$ E_2 \ll \sum_{p\leq r^{1/(2\sigma)}}\frac{r}{p^{\sigma}}+\sum_{r^{1/(2\sigma)}<p}\frac{r}{p^{2\sigma}}\ll r^{1/2+1/(2\sigma)}.$$
Now using Lemma 3.1 we find that
$$ \sum_{r^{1/\sigma}\log^A r<p}f_p\left(\frac{r}{p^{\sigma}}\right)\ll r^2\sum_{r^{1/\sigma}\log^A r<p} \frac{1}{p^{2\sigma}}\ll \frac{r^{1/\sigma}}{(\log r)^{1+A(2\sigma-1)}},$$
by the prime number theorem. Therefore we may assume that $y\leq r^{1/\sigma}(\log r)^{1/(2\sigma-1)}$, otherwise the error term corresponding to $y$ in Proposition 3.2 can be omitted. Since the sequence $\{Z(p)\}_p$ satisfies hypothesis (ULD), then for large primes $p$ we have  that $f_p(t)=g_Z(t)+O_A((\log p)^{-A}).$ Hence choosing $A=3/(2\sigma-1)$ gives that
\begin{equation*}
\begin{aligned}
\sum_{r^{1/(2\sigma)}<p\leq y}f_p\left(\frac{r}{p^{\sigma}}\right)
&= \sum_{r^{1/(2\sigma)}<p\leq y}g_Z\left(\frac{r}{p^{\sigma}}\right)+O\left(\frac{1}{(\log r)^A}\sum_{r^{1/(2\sigma)}<p\leq y}1\right)\\
&=\sum_{r^{1/(2\sigma)}<p\leq y}g_Z\left(\frac{r}{p^{\sigma}}\right)+O\left(\frac{r^{1/\sigma}}{\log^2 r}\right),\\
\end{aligned}
\end{equation*}
by the prime number theorem and our assumption on $y$.
Thus it only remains to evaluate the sum over $g_Z(r/p^{\sigma})$. To this end we use the prime number theorem
 in the form
$$\pi(t)=\int_2^t\frac{du}{\log u}+O\left(te^{-8\sqrt{\log t}}\right).$$
Moreover since the sequence $\{Z(p)\}_p$ converges in distribution to $Z$, then $Z$ has bounded support and $\ex(Z)=0$. Therefore by Lemma 3.1 and our hypothesis on $y$ we get that
$$ \sum_{r^{1/(2\sigma)}<p\leq y}g_Z\left(\frac{r}{p^{\sigma}}\right)=
\int_{r^{1/(2\sigma)}}^{y}g_Z\left(\frac{r}{t^{\sigma}}\right)d\pi(t)=
\int_{r^{1/(2\sigma)}}^{y}g_Z\left(\frac{r}{t^{\sigma}}\right)\frac{dt}{\log t}+E_3,$$
where
\begin{equation*}
\begin{aligned}
E_3&\ll g_Z\left(\frac{r}{y^{\sigma}}\right)ye^{-8\sqrt{\log y}}+g_Z\left(\sqrt{r}\right)r^{\frac{1}{2\sigma}}e^{-4\sqrt{\log r}}
+\int_{r^{\frac{1}{2\sigma}}}^{y}\frac{r}{t^{\sigma+1}}\left|g_Z'\left(\frac{r}{t^{\sigma}}\right)\right|te^{-8\sqrt{\log t}}dt.\\
&\ll \frac{r^2}{y^{2\sigma-1}}e^{-4\sqrt{\log r}}+r^{1/\sigma}e^{-4\sqrt{\log r}}+ re^{-4\sqrt{\log r}}
\left(\int_{(rd)^{\frac{1}{2\sigma}}}^{y}\frac{1}{t^{\sigma}}dt\right),\\
&\ll r^{1/\sigma}e^{-\sqrt{\log r}},
\end{aligned}
\end{equation*}
To estimate the main term we make the change of variables $u=r/t^{\sigma}$. This gives
$$
\int_{r^{1/(2\sigma)}}^{y}g_Z\left(\frac{r}{t^{\sigma}}\right)\frac{dt}{\log t}=
r^{1/\sigma}\int_{r/y^{\sigma}}^{r^{1/2}} \frac{g_Z(u)}{u^{1+\frac{1}{\sigma}}\log(r/u)}du.$$
In the range $r/y^{\sigma}\leq u\leq r^{1/2}$, we have
$$ \frac{1}{\log(r/u)}=\frac{1}{\log r}\frac{1}{1-\frac{\log u}{\log r}}=\frac{1}{\log r}+O\left(\frac{\log u}{\log^2 r}\right),$$
which implies that
$$ \int_{r/y^{\sigma}}^{r^{1/2}} \frac{g_Z(u)}{u^{1+\frac{1}{\sigma}}\log(r/u)}du=
\frac{1}{\log r} \int_{r/y^{\sigma}}^{r^{1/2}} \frac{g_Z(u)}{u^{1+\frac{1}{\sigma}}}du+O\left(\frac{1}{\log^2r}\right),$$
using that
$$ \int_{0}^{\infty}\frac{g_Z(u)\log(u)}{u^{1+\frac{1}{\sigma}}}du\ll 1,$$
which follows from Lemma 3.1.
Using Lemma 3.1 again gives that
$$ \int_{r/y^{\sigma}}^{r^{1/2}} \frac{g_Z(u)}{u^{1+\frac{1}{\sigma}}}du=\int_{0}^{\infty}\frac{g_Z(u)}{u^{1+\frac{1}{\sigma}}}du+O\left(r^{1/2-1/(2\sigma)}+\left(\frac{r^{1/\sigma}}{y}\right)^{2\sigma-1}\right).$$
Hence we deduce that
$$ \sum_{r^{1/(2\sigma)}<p\leq y}g_Z\left(\frac{r}{p^{\sigma}}\right)
= \frac{r^{1/\sigma}}{\log r}\int_{0}^{\infty}\frac{g_Z(u)}{u^{1+\frac{1}{\sigma}}}du\left(1+O\left(\frac{1}{\log r}+\left(\frac{r^{1/\sigma}}{y}\right)^{2\sigma-1}\right)\right).
$$
Finally if $Z$ is symmetric then $g_Z(u)$ is even, and hence we get the same asymptotic if $r$ is replaced by $-r$. This concludes the proof.
\end{proof}

\begin{proof}[Proof of Proposition 3.3] For a prime $p$ let $h_p(t):=\log\ex(e^{tY(p)})$ be the cumulant-generating function of $Y(p)$.
 We follow the same lines as the proof of Proposition 3.2. Define $$E^i_p(r):=\ex\left(\prod_{j=1}^d
\left(1-\frac{X_j(p)}{p^{\sigma}}\right)^{ir/2}\left(1-\frac{\overline{X_j(p)}}{p^{\sigma}}\right)^{-ir/2}\right).$$
 The independence of the $\X(p)$ implies that $\ex\left(L(\sigma,\X;y)^{-i r/2}\overline{L(\sigma,\X;y)}^{ir/2}\right)=\prod_{p\leq y}E^i_p(r).$
 If $p\leq r^{1/(2\sigma)}$, then $\log E^i_p(r)=O(r/p^{\sigma})$, since the $X_j(p)$ are bounded. Now for primes $p$ such that $r^{1/{2\sigma}}<p\leq y
$, we have that
$$
\prod_{j=1}^d
\left(1-\frac{X_j(p)}{p^{\sigma}}\right)^{ir/2}\left(1-\frac{\overline{X_j(p)}}{p^{\sigma}}\right)^{-ir/2}
=\exp\left(\frac{r}{p^{\sigma}}Y(p)\right)\left(1+
O\left(\frac{r}{p^{2\sigma}}\right)\right),
$$
so that $\log E^i_p(r)=h_p\left(r/p^{\sigma}\right)+O\left(r/p^{2\sigma}\right)$. Then following exactly the same method as in the proof of Proposition 3.2 gives the result.
 \end{proof}

\begin{proof}[Proof of Theorem 1.9] We begin by estimating $\Phi_X(\tau;y)$. For $s>0$ we have
\begin{equation*}
\begin{aligned}
&\int_{-\infty}^{\infty}se^{st}\Phi_X(t;y)dt =
\int_{-\infty}^{\infty}se^{st}\int_{\log|L(\sigma,\X(\omega);y)|>t}d\mu(\omega)dt\\
&=\int_{\Omega}|L(\sigma,\X(\omega);y)|^{s}d\mu(\omega)=\ex\left(|L(\sigma,\X;y)|^{s}\right).\\
\end{aligned}
\end{equation*}
Therefore if $s$ is large, then Proposition 3.2 gives that
\begin{equation}
\int_{-\infty}^{\infty}e^{st}\Phi_X(t;y)dt=\exp\left(G_Z(\sigma)\frac{s^{1/\sigma}}{\log s}\left(1+O\left(\frac{1}{\log s}+\left(\frac{s^{1/\sigma}}{y}\right)^{2\sigma-1}\right)\right)\right).
\end{equation}
To estimate $\Phi_X(\tau;y)$ we use the saddle point method. Let $s$ be the unique solution to the equation
\begin{equation}
\tau=G_Z(\sigma)\frac{s^{1/\sigma-1}}{\sigma\log s}.
\end{equation}
Let $\epsilon>0$ be a small number to be chosen later and define $$s_1:=s(1+\epsilon), \ s_2:=s(1-\epsilon),\text{ and }
\tau_1:=\tau\left(1+\frac{\epsilon}{2\sigma}\right), \ \tau_2:=\tau\left(1-\frac{\epsilon}{2\sigma}\right).$$
Since $s-s_2>0$, then
$$ \int_{-\infty}^{\tau_2}e^{st}\Phi_X(t;y)dt\leq \int_{-\infty}^{\tau_2}e^{(s-s_2)(\tau_2-t)+st}\Phi_X(t;y)dt\leq e^{\epsilon s\tau_2}\int_{-\infty}^{+\infty}e^{s_2t}\Phi_X(t;y)dt.$$
Hence using (3.1) we find that
\begin{equation}
\int_{-\infty}^{\tau_2}e^{st}\Phi_X(t;y)dt\leq \exp\left(G_Z(\sigma)\frac{s^{1/\sigma}}{\log s}\left((1-\epsilon)^{1/\sigma}+\frac{\epsilon}{\sigma}-\frac{\epsilon^2}{2\sigma^2}+E_4\right)\right),
\end{equation}
where $E_4\ll 1/\log s+\left(s^{1/\sigma}y^{-1}\right)^{2\sigma-1}.$ Now we have that $(1+x)^{1/\sigma}=1+\sigma^{-1}x+x^2\sigma^{-1}(\sigma^{-1}-1)/2+O(x^3)$ if $|x|<1$. Then we choose
$$ \epsilon=K\left(\frac{1}{\sqrt{\log s}}+\left(\frac{s^{1/\sigma}}{y}\right)^{\sigma-1/2}\right),$$
where $K$ is a suitably large constant, to deduce that
\begin{equation}
\int_{-\infty}^{\tau_2}e^{st}\Phi_X(t;y)dt\leq \frac{1}{10}\int_{-\infty}^{\infty}e^{st}\Phi_X(t;y)dt.
\end{equation}
Similarly one has
$$ \int_{\tau_1}^{+\infty}e^{st}\Phi_X(t;y)dt\leq \int_{\tau_1}^{+\infty}e^{(s_1-s)(t-\tau_1)+st}\Phi_X(t;y)dt\leq e^{-\epsilon s\tau_1}\int_{-\infty}^{+\infty}e^{s_1t}\Phi_X(t;y)dt,$$
and using exactly the same argument as before we deduce that
\begin{equation}
\int_{\tau_1}^{+\infty}e^{st}\Phi_X(t;y)dt\leq \frac{1}{10}\int_{-\infty}^{\infty}e^{st}\Phi_X(t;y)dt.
\end{equation}
Combining inequalities (3.4) and (3.5) along with the estimate (3.1) we obtain that
$$  \int_{\tau_2}^{\tau_1}e^{st}\Phi_X(t;y)dt= \exp\left(G_Z(\sigma)\frac{s^{1/\sigma}}{\log s}\left(1+O(\epsilon^2)\right)\right).$$
Moreover, since $\Phi_X(t;y)$ is a non-increasing function and $\int_{\tau_2}^{\tau_1}e^{st}dt=\exp(s\tau(1+O(\epsilon)))$, we get that
$$ \Phi_X\left(\tau\left(1+\frac{\epsilon}{2\sigma}\right);y\right)\leq \exp\left(-\frac{(1-\sigma)G_Z(\sigma)}{\sigma}\frac{s^{1/\sigma}}{\log s}(1+O(\epsilon))\right)\leq \Phi_X\left(\tau\left(1-\frac{\epsilon}{2\sigma}\right);y\right).$$
Hence it only remains to solve equation (3.2) in $s$. Taking the logarithm of both sides we get that $\log s= \frac{\sigma}{(1-\sigma)}\log \tau +O(\log_2\tau).$ Then an easy calculation gives that
$$ s=\left(\frac{\sigma^2}{(1-\sigma)G_Z(\sigma)}(\tau\log \tau)\right)^{\frac{\sigma}{(1-\sigma)}}\left(1+O\left(\frac{\log_2\tau}{\log \tau}\right)\right).$$
Thus we deduce that
$$ \Phi_X\left(\tau;y\right)=\exp\left(-A_Z(\sigma)\tau^{\frac{1}{1-\sigma}}(\log\tau)^{\frac{\sigma}{1-\sigma}}\left(1+O\left(\frac{1}{\sqrt{\log \tau}}+r(y,\tau)\right)\right)\right).$$
Now concerning $\Psi_X(\tau;y)$ we have for $s>0$
\begin{equation*}
\begin{aligned}
&\int_{-\infty}^{\infty}se^{st}\Psi_X(t;y)dt =
\int_{-\infty}^{\infty}se^{st}\int_{\arg L(\sigma,\X(\omega);y)>t}d\mu(\omega)dt\\
&=\int_{\Omega}e^{s\arg L(\sigma,\X(\omega);y)}d\mu(\omega)=\ex\left(L(\sigma,\X;y)^{-is/2}{\overline{ L(\sigma,\X;y)}}^{is/2}\right).\\
\end{aligned}
\end{equation*}
Then by Proposition 3.3 we can use exactly the same saddle-point method as for $\Phi_X(\tau;y)$ to derive the analogous estimate for $\Psi_X(\tau;y)$. Finally to get estimates for the left tails we proceed along the same lines by changing $s$ to $-s$.
\end{proof}

\section{The distribution of $\zeta(\sigma+it)$ and $L(\sigma,\chi)$}

\subsection{The distribution of the Riemann zeta function}
Define $\zeta(s;y):=\prod_{p\leq y}(1-p^{-s})^{-1}$, and $L(\sigma,\X_1;y):=\prod_{p\leq y}(1-X_1(p)p^{-\sigma})^{-1}$, where $
\{X_1(p)\}_p$ are independent random variables uniformly distributed on the unit circle. In order to prove Theorem 1.1 our strategy consists of using zeros density estimates for $\zeta(s)$, Lemma 2.2 along with a basic ``large sieve'' (Lemma 4.2 below) to show that one can approximate
$\log \zeta(\sigma+it)$ by $\log\zeta(\sigma+it;y)$ (where $y=\log T$) for all $t\in [T,2T]$ except for a set of a very small measure.
Then we compute large moments of $\zeta(\sigma+it;y)$ (Proposition 4.1 below) to show that the distribution of $\log\zeta(\sigma+it;y)$ is very close to that of $\log L(\sigma,\X_1;y)$, and that the latter can be deduced from the results of section 3. We prove
\begin{pro} Let $T$ be large, and $y\leq (\log T)^2$ be a large real number. Then uniformly for all complex numbers $z_1, z_2$ such that $|z_i|y^{1-\sigma}\leq (1-\sigma)\log T/16$ we have
\begin{equation*}
\begin{aligned}
\frac{1}{T}\int_{T}^{2T}\zeta(\sigma+it;y)^{z_1}\overline{\zeta(\sigma+it;y)}^{z_2}dt&= \ex\left(L(\sigma,\X_1;y)^{z_1}\overline{L(\sigma,\X_1;y)}^{z_2}\right)\\
&+O\left(\exp\left(-\frac{\log T}{4\log y}\right)\right).\\
\end{aligned}
\end{equation*}
\end{pro}
\begin{proof} We have that
$$
\frac{1}{T}\int_{T}^{2T}\zeta(\sigma+it,y)^{z_1}\overline{\zeta(\sigma+it,y)}^{z_2}dt=\sum_{m,n\in S(y)}\frac{d_{z_1}(n)d_{z_2}(m)}{(mn)^{\sigma}}\frac{1}{T}\int_{T}^{2T}\left(\frac{m}{n}\right)^{it}dt.
$$
The contribution of the diagonal terms $m=n$ equals
$\sum_{n\in S(y)}d_{z_1}(n)d_{z_2}(n)/n^{2\sigma}=\ex\left(L(\sigma,\X_1;y)^{z_1}\overline{L(\sigma,\X_1;y)}^{z_2}\right)$ by equation (2.4). This contribution constitutes the main term to the moments as we shall now prove.
Let $k$ be the smallest positive integer such that $k\geq \max(|z_1|,|z_2|)$. Concerning the off-diagonal terms $m\neq n$, we split these into two cases. First we handle the terms $m, n\leq T^{1/2}$. In this case observe that
$\int_{T}^{2T}\left(\frac{m}{n}\right)^{it}dt\ll 1/|\log(m/n)|\ll T^{1/2}.$
Hence by (2.1) it follows that the contribution of these terms is
$$ \ll \frac{1}{\sqrt{T}}\left(\sum_{n\in S(y)}\frac{d_k(n)}{n^{\sigma}}\right)^2=\exp\left(-\frac{\log T}{2}+O\left(\frac{ky^{1-\sigma}}{\log y}\right)\right).$$
Next we handle the terms $m\neq n$ with $\max(m,n)>\sqrt{T}.$
 Let $\beta=1/\log y$. By (2.1) the contribution of these terms is
\begin{equation*}
\begin{aligned}
&\ll \sum_{\substack{ m>\sqrt{T}\\ m\in S(y)}} \sum_{n\in S(y)}\frac{d_k(n)d_k(m)}{(mn)^{\sigma}}\ll T^{-\beta/2}\sum_{m\in S(y)}\frac{d_k(m)}{m^{\sigma-\beta}}\sum_{n\in S(y)}\frac{d_k(n)}{n^{\sigma}}\\
&\ll \exp\left(-\frac{\log T}{2\log y}+ \frac{(e+1)ky^{1-\sigma}}{(1-\sigma)\log y}+O\left(\frac{ky^{1-\sigma}}{(1-\sigma)^2\log^2 y}\right)\right),\\
\end{aligned}
\end{equation*}
 which completes the proof.
\end{proof}

\begin{lem} Let $2\leq y\leq z$ be real numbers. Then for all positive integers $k$ with $1\leq k\leq \log T/(3\log z)$ we have
$$ \frac{1}{T}\int_T^{2T}\left|\sum_{y\leq p\leq z}\frac{1}{p^{\sigma+it}}\right|^{2k}dt\ll k!\left(\sum_{y\leq p\leq z}\frac{1}{p^{2\sigma}}\right)^k+ O\left(T^{-1/3}\right) .$$
\end{lem}
\begin{proof}
First we have that
$$\frac{1}{T}\int_T^{2T}\left|\sum_{y\leq p\leq z}\frac{1}{p^{\sigma+it}}\right|^{2k}dt= \sum_{\substack{ y\leq p_1,...,p_k\leq z\\y\leq q_1,...,q_k\leq z}}\frac{1}{(p_1\cdots p_kq_1\cdots q_k)^{\sigma}}\frac{1}{T}\int_T^{2T}\left(\frac{p_1\cdots p_k}{q_1\cdots q_k}\right)^{it}dt.$$
The diagonal terms $p_1\cdots p_k=q_1\cdots q_k$ contributes
$$\ll k!\left(\sum_{y\leq p\leq z}\frac{1}{p^{2\sigma}}\right)^k.$$
If $p_1\cdots p_k\neq q_1\cdots q_k$ then both products are below $z^k\leq T^{1/3}$, which implies that
$$ \frac{1}{T}\int_T^{2T}\left(\frac{p_1\cdots p_k}{q_1\cdots q_k}\right)^{it}dt\ll \frac{1}{T|\log((p_1\cdots p_k)/(q_1\cdots q_k))|}\ll T^{-2/3}.$$
Therefore the contribution of the off-diagonal terms is $T^{-\frac23}\left(\sum_{y\leq p\leq z}p^{-\sigma}\right)^{2k}\ll T^{-2/3}z^{2k(1-\sigma)} \ll T^{-1/3}$. \qed
\end{proof}

\begin{proof}[Proof of Theorem 1.1] Let $1/2<\sigma<1$ and take $y=\log T$. For simplicity we write $\Phi_T(\tau)=\Phi_T(\sigma,\tau)$.
 Proposition 4.1 (with $z_1=z_2=r/2$) implies that
\begin{equation}
\frac{1}{T}\int_T^{2T}|\zeta(\sigma+it;y)|^rdt=\ex\left(|L(\sigma,\X_1;y)|^r\right)+O\left(\exp\left(-\frac{\log T}{4\log y}\right)\right),
\end{equation}
uniformly for all real numbers $r$ in the range  $r\leq (1-\sigma)(\log T)^{\sigma}/8.$ Let
$$ \Phi_T(\tau;y):=\frac{1}{T}\text{meas}\{t\in [T,2T]: \log|\zeta(\sigma+it;y)|>\tau\}.$$ Then using (4.1) along with Proposition 3.2 (with $d=1$ and $\X(p)=X_1(p)$) gives that
\begin{equation}
\begin{aligned}
\int_{-\infty}^{\infty}re^{ru}\Phi_T(u;y)du&=\frac{1}{T}\int_T^{2T}|\zeta(\sigma+it;y)|^rdt\\
&=\exp\left(G_1(\sigma)
\frac{r^{1/\sigma}}{\log r}\left(1+O\left(\frac{1}{\log r}+\left(\frac{r^{1/\sigma}}{\log T}\right)^{2\sigma-1}\right)\right)\right),\\
\end{aligned}
\end{equation}
where $G_1(\sigma)=\int_{0}^{\infty}\log I_0(u)u^{-1-\frac{1}{\sigma}}du.$
In order to estimate $\Phi_T(\tau;y)$ we use the saddle point method exactly as in the proof of Theorem 1.9 (see section 3).
In this case $r$ will be chosen to be the unique solution to the equation $\tau=G_1(\sigma)r^{1/\sigma-1}/(\sigma\log r)$ (see 3.2), which implies that
$$ r=\left(\frac{\sigma^2}{(1-\sigma)G_1(\sigma)}(\tau\log \tau)\right)
^{\frac{\sigma}{(1-\sigma)}}\left(1+O\left(\frac{\log_2\tau}{\log \tau}\right)\right).$$
Therefore, choosing $c_1(\sigma)$ small enough and applying the saddle point method to equation (4.2), we deduce that uniformly for $1\ll\tau\leq c_1(\sigma)(\log T)^{1-\sigma}/\log_2 T$, we have that
\begin{equation}
\Phi_T\left(\tau;y\right)=\exp\left(-A_1(\sigma)\tau^{\frac{1}{(1-\sigma)}}(\log\tau)^{\frac{\sigma}{(1-\sigma)}}\left(1+O\left(\frac{1}{\sqrt{\log\tau}}
+r(\log T,\tau)\right)\right)\right).
\end{equation}
Therefore what remains is to show that $\log\Phi_T(\tau)$ has the same asymptotic formula as $\log\Phi_T(\tau;y)$, in our range of $\tau$. To this end we will construct a set $\mathcal{A}(T,\tau)\subset [T,2T]$ with very small measure (negligible compared to $T\Phi_T(\tau;y)$) such that $\log\zeta(\sigma+it)\approx\log\zeta(\sigma+it;y)$ for $t\in [T,2T]\setminus \mathcal{A}(T,\tau)$.

Let $N(\sigma_0,T)$ denote the number of zeros of $\zeta(s)$ in the rectangle $\{\text{Re}(s)>\sigma_0, |\text{Im}(s)|\leq T\}$. Then using the zeros-density result $N(\sigma_0,T)\ll T^{3/2-\sigma_0}(\log T)^5$ (see Theorem 9.19 A of [27]) along with Lemma 2.2 with $z=(\log T)^{3/(\sigma-1/2)}$, and $\sigma_0=\sigma/2+1/4>1/2$, it follows that
\begin{equation}
\log\zeta(\sigma+it)=\sum_{n=2}^z\frac{\Lambda(n)}{n^{\sigma+it}\log n}+O\left(\frac{1}{(\log T)^{1/4}}\right),
\end{equation}
for all $t\in [T,2T]$ except for a set $\mathcal{A}_0(T)$ with measure $\ll T^{1-(\sigma-1/2)/4}$. Since $\tau\ll (\log T)^{1-\sigma}/\log_2 T$, it follows from (4.3) that
\begin{equation}
\frac{1}{T}\text{meas}\mathcal{A}_0(T)=o(\Phi_T(2\tau;y)).
\end{equation}
  Moreover, we have that
\begin{equation}
\sum_{n=2}^z\frac{\Lambda(n)}{n^{\sigma+it}\log n}= \log\zeta(\sigma+it;y)+\sum_{y\leq p\leq z}\frac{1}{p^{\sigma +it}}+O\left(\frac{1}{(\log T)^{\sigma-1/2}}\right).
\end{equation}
Now let $\mathcal{A}_1(T,\tau,\epsilon)$ be the set of values $t\in [T,2T]$ such that $|\sum_{y\leq p\leq z}1/p^{\sigma+it}|>\epsilon\tau$. Then using Lemma 4.2 we have that
$$\text{meas}\mathcal{A}_1(T,\tau,\epsilon)\leq (\epsilon\tau)^{-2k}\int_T^{2T}\left|\sum_{y\leq p\leq z}\frac{1}{p^{\sigma+it}}\right|^{2k}dt\ll T\left(\frac{2k}{(2\sigma-1)\epsilon^2\tau^2y^{2\sigma-1}\log y}\right)^k,$$
for all integers $1\leq k\leq (\sigma-1/2)\log T/(9\log_2 T)$. We choose $\epsilon=Cr(\log T,\tau)$, where $C$ is a suitably large constant. Remark that $r(s,\tau)^2\tau^2s^{2\sigma-1}=\tau^{1/(1-\sigma)}(\log s)^{\frac{(2\sigma-1)}{(1-\sigma)}}$. Then with this choice of $\epsilon$ and if $c_1(\sigma)$ is small enough, we may choose $$k=[((2\sigma-1)\epsilon^2\tau^2(\log T)^{2\sigma-1}\log_2 T)/10]$$ to get that
\begin{equation}
\frac{1}{T}\text{meas}\mathcal{A}_1(T,\tau,\epsilon)
\ll \exp\left(-\frac{C^2(2\sigma-1)}{10}\tau^{1/(1-\sigma)}(\log_2 T)^{\sigma/(1-\sigma)}\right).
\end{equation}
Therefore if $C$ is large enough, it follows from (4.3) and (4.7) that
\begin{equation}
\frac{1}{T}\text{meas}\mathcal{A}_1(T,\tau,\epsilon)=o(\Phi_T(2\tau;y)).
\end{equation}
Now let $\mathcal{A}(T,\tau):=\mathcal{A}_0(T)\cup \mathcal{A}_1(T,\tau,\epsilon).$
Then by (4.4) and (4.6) we have that
\begin{equation} |\log \zeta(\sigma+it)-\log\zeta(\sigma+it;y)|<\delta\tau, \text{ where } \delta=\epsilon+\frac{1}{\log\tau},
\end{equation}
for all $t\in [T,2T]\setminus \mathcal{A}(T,\tau).$
This implies that
$$\Phi_T\left(\tau(1+\delta);y\right)+O\left(\frac{\text{meas}\mathcal{A}(T,\tau)}{T}\right)\leq \Phi_T(\tau)\leq \Phi_T\left(\tau(1-\delta);y\right)+O\left(\frac{\text{meas}\mathcal{A}(T,\tau)}{T}\right).$$
The result then follows upon combining (4.3), (4.5) and (4.8).

Similarly let $\Psi_T(\tau)=\frac{1}{T}\text{meas}\{t\in [T,2T]: \arg\zeta(\sigma+it)>\tau\}$ and  $\Psi_T(\tau;y)=\frac{1}{T}\text{meas}\{t\in [T,2T]: \arg\zeta(\sigma+it;y)>\tau\}.$ Then by Proposition 4.1 we have that
\begin{equation*}
\begin{aligned}
\int_{-\infty}^{\infty}re^{rt}\Psi_T(\tau;y)dt&=\frac{1}{T}\int_T^{2T}\zeta(\sigma+it;y)^{-ir/2}\zeta(\sigma-it;y)^{ir/2}dt\\
&=\ex\left(L(\sigma,\X_1;y)^{-ir/2}\overline{L(\sigma,\X_1;y)}^{ir/2}\right)+O\left(\exp\left(-\frac{\log T}{4\log y}\right)\right),\\
\end{aligned}
\end{equation*}
 for $r\leq (1-\sigma)(\log T)^{\sigma}/8$. Then appealing to Proposition 3.3 and using the saddle-point method as in the proof of Theorem 1.9, we can deduce that $\Psi_T(\tau;y)$ has the same asymptotic as (4.3). This is due to the fact that $\ex\left(e^{t\text{Re}X}\right)=\ex\left(e^{t\text{Im}X}\right)=I_0(t)$ for a random variable $X$ uniformly distributed on the unit circle. Finally we should note that the last part of the argument to estimate $\Psi_T(\tau)$ is the same as for $\Phi_T(\tau)$ using the same choice of the parameters $k$ and $\epsilon$, since the inequality (4.9) does also control the difference  $|\arg\zeta(\sigma+it)-\arg\zeta(\sigma+it;y)|$. The procedure is also analogous for the left tails of $\log|\zeta(\sigma+it)|$ and $\arg\zeta(\sigma+it)$, changing $r$ to $-r$.
\end{proof}
\subsection{ The distribution of Dirichlet $L$-functions}

In order to apply the same method (as in the case of $\zeta(\sigma+it)$)  and derive similar results for the family $\{L(\sigma,\chi): \chi\neq \chi_0 \ (\text{mod } q)\}$, we need to compute asymptotics for complex moments of short Euler products $L(\sigma,\chi;y):=\prod_{p\leq y}\left(1-\chi(p)p^{-\sigma}\right)^{-1}$ (analogue of Proposition 4.1) and prove the analogue of Lemma 4.2.
We prove
\begin{pro} Let $q$ be a large prime, and $y\leq (\log q)^2$ be a large real number. Then
uniformly for all complex numbers $z_1, z_2$ such that
 $|z_i|y^{1-\sigma}\leq (1-\sigma)\log q/8$ we have
\begin{equation*}
\begin{aligned}
\frac{1}{\phi(q)}\sum_{\substack{\chi (\text{mod } q)\\ \chi\neq\chi_0}} L(\sigma,\chi;y)^{z_1}\overline{L(\sigma,\chi;y)}^{z_2}&= \ex\left(L(\sigma,\X_1;y)^{z_1}\overline
{L(\sigma,\X_1;y)}^{z_2}\right)\\
&+O\left(\exp\left(-\frac{\log q}{2\log y}\right)\right).\\
\end{aligned}
\end{equation*}
\end{pro}

\begin{proof}[Proof]
  Let $k$ be the smallest integer with $k\geq \max(|z_1|,|z_2|)$. Then
\begin{equation*}
\begin{aligned}
&\frac{1}{\phi(q)}\sum_{\substack{\chi (\text{mod } q)\\ \chi\neq\chi_0}} L(\sigma,\chi;y)^{z_1}\overline{L(\sigma,\chi;y)}^{z_2}=\frac{1}{\phi(q)}\sum_{\substack{\chi (\text{mod } q)\\\chi\neq\chi_0}} \chi(n)\overline{\chi(m)}\sum_{m,n\in S(y)} \frac{d_{z_1}(n)d_{z_2}(m)}{(mn)^{\sigma}}\\
&=\frac{1}{\phi(q)}\sum_{\chi (\text{mod } q)}\chi(n)\overline{\chi(m)}\sum_{m,n\in S(y)} \frac{d_{z_1}(n)d_{z_2}(m)}{(mn)^{\sigma}}+O\left(\frac{1}{q}\left(\sum_{n\in S(y)}\frac{d_k(n)}{n^{\sigma}}\right)^2\right).\\
&=\sum_{\substack{ m,n\in S(y)\\ m\equiv n\text{ mod } q}} \frac{d_{z_1}(n)d_{z_2}(m)}{(mn)^{\sigma}}+O\left(\exp\left(-\log q+ \frac{3ky^{1-\sigma}}{(1-\sigma)\log y}\right)\right),
\end{aligned}
\end{equation*}
using equation (2.1) along with the orthogonality relation for characters. The contribution of the diagonal terms $m=n$ equals
$$\sum_{n\in S(y)}\frac{d_{z_1}(n)d_{z_2}(n)}{n^{2\sigma}}=\ex\left(L(\sigma,\X_1;y)^{z_1}\overline{L(\sigma,\X_1;y)}^{z_2}\right).$$
Since $m\equiv n \text{ mod } q$, the off-diagonal terms $m\neq n$ must satisfy  $\max(m,n)>q$ . Put $\beta=\frac{1}{\log y}$. Then by (2.1) the contribution of these terms is bounded by
\begin{equation*}
\begin{aligned}
2\sum_{\substack{ n\in S(y)\\ n>q}}\frac{d_k(n)}{n^{\sigma}}\sum_{m\in S(y)} \frac{d_k(m)}{m^{\sigma}}
&\leq 2q^{-\beta}\sum_{m\in S(y)}\frac{d_k(m)}{m^{\sigma}}\sum_{n\in S(y)}\frac{d_k(m)}{m^{\sigma-\beta}}.\\
&\ll \exp\left(-\frac{\log q}{\log y}+\frac{(e+1)ky^{1-\sigma}}{(1-\sigma)\log y}+O\left(\frac{ky^{1-\sigma}}{(1-\sigma)^2\log^2 y}\right)\right),\\
\end{aligned}
\end{equation*}
which completes the proof.
\end{proof}

\begin{lem} Let $q$ be a large prime and $2\leq y\leq z$ be real numbers. Then for all positive integers $k$ such that $1\leq k\leq \log q/(2\log z)$ we have
$$ \frac{1}{\phi(q)}\sum_{\substack{\chi (\text{mod } q)\\ \chi\neq\chi_0}}\left|\sum_{y\leq p\leq z}\frac{\chi(p)}{p^{\sigma}}\right|^{2k}\ll k! \left(\sum_{y\leq p\leq z}\frac{1}{p^{2\sigma}}\right)^k +O\left(q^{-1/2}\right).$$
\end{lem}

\begin{proof}
We have that
$$ \frac{1}{\phi(q)}\sum_{\substack{\chi (\text{mod } q)\\ \chi\neq\chi_0}}\left|\sum_{y\leq p\leq z}\frac{\chi(p)}{p^{\sigma}}\right|^{2k}=\frac{1}{\phi(q)}\sum_{\substack{\chi (\text{mod } q)\\\chi\neq\chi_0}}\sum_{y\leq p_1,...,p_{2k}\leq z}\frac{\chi(p_1\cdots p_k)\overline{\chi}(p_{k+1}\cdots p_{2k})}{(p_1\cdots p_{2k})^{\sigma}}.$$
The contribution of the diagonal terms $p_1\cdots p_k=p_{k+1}\cdots p_{2k}$ is
$$ \ll k! \left(\sum_{y\leq p\leq z}\frac{1}{p^{2\sigma}}\right)^k.$$
Now if $p_1\cdots p_k\neq p_{k+1}\cdots p_{2k}$, then $ \sum_{\chi\neq\chi_0}\chi(p_1\cdots p_k)\overline{\chi}(p_{k+1}\cdots p_{2k})=-1$ since

\noindent  $p_1\cdots p_k, p_{k+1}\cdots p_{2k}\leq z^k<q.$ Therefore the contribution of these terms is

\noindent $q^{-1}\left(\sum_{y\leq p\leq z}p^{-\sigma}\right)^{2k}\ll q^{-1}z^{2k(1-\sigma)} \ll q^{-1/2}$.
\end{proof}
Let $q$ be a large prime and define $$\Phi^{\text{char}}_{q}(\sigma,\tau):=\frac{1}{\phi(q)}|\{\chi\neq \chi_0, \chi \ (\text{mod } q): \log|L(\sigma,\chi)|>\tau\}|.$$
Then using the same method as in the proof of Theorem 1.1 we derive
\begin{thm}
Let $1/2<\sigma<1$, and $q$ be a large prime. Then there exists $c_4(\sigma)>0$ such that uniformly in the range $1\ll \tau\leq c_4(\sigma)(\log q)^{1-\sigma}/\log_2 q$ we have
$$ \Phi^{\text{char}}_{q}(\sigma,\tau)=\exp\left(-A_1(\sigma)\tau^{\frac{1}{(1-\sigma)}}(\log\tau)^{\frac{\sigma}{(1-\sigma)}}\left(1+O\left(\frac{1}{\sqrt{\log\tau}}+r(\log q,\tau)\right)\right)\right).
$$
This estimate also holds for the proportion of non-principal characters $\chi \ (\text{mod }q)$ such that $\arg L(s,\chi)>\tau$.
\end{thm}
\begin{proof}
For simplicity write $\Phi_q(\tau)=\Phi^{\text{char}}_{q}(\sigma,\tau)$. The result can be deduced by proceeding along the same lines as in the proof of Theorem 1.1. Indeed all the parameters will be chosen exactly by changing $T$ to $q$. Let $y=\log q$ and define
$\Phi_q(\tau;y)$ to be the proportion of characters $\chi\neq \chi_0 (\text{mod } q)$ such that $\log|L(\sigma,\chi;y)|>\tau.$
Then for all positive real numbers $r\leq (1-\sigma)(\log q)^{\sigma}/8,$ Proposition 4.3 gives that
$$
\int_{-\infty}^{\infty}re^{rt}\Phi_q(t;y)dt=\frac{1}{\phi(q)}\sum_{\substack{\chi (\text{mod } q)\\ \chi\neq\chi_0}}|L(\sigma,\chi;y)|^r=\ex\left(|L(\sigma,\X_1;y)|^{r}\right)+o(1).
$$
Then using Proposition 3.2 and the saddle point method (as in the proof of Theorem 1.9) we deduce that $\Phi_q(\tau;y)$ as the same asymptotic as $\Phi_T(\tau;y)$ (see equation (4.3)).
Therefore it only remains to construct a set $\mathcal{A}(q,\tau)$ which will play an similar role to that of $\mathcal{A}(T,\tau)$ in the proof of Theorem 1.1.
Let $N(\sigma,T, \chi)$ denotes the number of zeros of $L(s,\chi)$ such that Re$(s)\geq \sigma$ and $|\text{Im}(s)|\leq T$. We use the following zero-density result of Montgomery [18] which states that for $T\geq 2$ and $1/2<\sigma<1$  we have
$ \sum_{\chi \ (\text{mod } q)} N(\sigma,T, \chi) \ll (qT)^{3(1-\sigma)/(2-\sigma)}(\log qT)^{14}$. Using this result along with Lemma 2.3 with $t=0$,
 $z=(\log q)^{3/(\sigma-1/2)}$ and $\sigma_0=\sigma/2+1/4>1/2$, gives that
$$ \log L(\sigma,\chi)=\sum_{n=2}^z\frac{\Lambda(n)\chi(n)}{n^{\sigma}}+O\left(\frac{1}{(\log q)^{1/4}}\right),$$
for all characters $\chi \ (\text{mod }q)$ except for a set $\mathcal{A}_0(q)$ of cardinality $\leq q^{1-a(\sigma)}$
for some constant $a(\sigma)>0$. Now we choose $\epsilon=Cr(\log q,\tau)$ where $C$ is a suitably large constant, and $\mathcal{A}_1(q,\tau,\epsilon)$ to be the set of characters such that
 $|\sum_{y\leq p\leq z}\chi(p)/p^{\sigma+it}|>\epsilon\tau$. Then Lemma 4.4 insures
that $|\mathcal{A}_1(q,\tau,\epsilon)|/\phi(q)=o(\Phi_q(2\tau;y))$, if $c_4(\sigma)$ is suitably small. Finally taking $\mathcal{A}(q,\tau)=\mathcal{A}_0(q)\cup\mathcal{A}_1(q,\tau,\epsilon)$,
 we see that $|\log L(\sigma,\chi)-\log L(\sigma,\chi;y)|<\delta\tau$, for all characters $\chi \notin \mathcal{A}(q,\tau)$ where $\delta=\epsilon+1/\log\tau$; and that $ |\mathcal{A}(q,\tau)|/\phi(q)=o(\Phi_q(2\tau;y))$. This gives the desired asymptotic for $\Phi_q(\tau)$, and one can handle the left tail of $\log |L(\sigma,\chi)|$ similarly. The analogous result for $\arg L(\sigma,\chi)$ follows along the same lines.
\end{proof}

\section{Distribution and extreme values of $L(\sigma,\chi_d)$}

\subsection{Distribution of values of $L(\sigma,\chi_d)$: proof of Theorem 1.6}

Let us first describe the probabilistic random model attached to this family. Let $\{X_2(p)\}$ be independent random variables taking the values $1$ and $-1$ with equal probability $p/(2(p+1))$ and the value $0$ with probability $1/(p+1)$. Then define
$$ L(\sigma,\X_2):=\prod_{p\leq y}\left(1-\frac{X_2(p)}{p^{\sigma}}\right)^{-1}.$$
This model was first introduced by Granville and Soundararajan [8] for $\sigma=1$. The reason for this choice over the simpler $\pm 1$ with probability $1/2$ (which was previously considered by many people including Chowla-Erd\"os, Elliott, and Montgomery-Vaughan) is that for odd primes $p$, fundamental discriminants $d$ lie in one of $p^2-1$ residue classes mod $p^2$ so that $\chi_d(p)=0$ for $p-1$ of these classes, and the remaining $p(p-1)$ residue classes split equally into $\pm 1$ values (for $p=2$ one can check that the values $0,\pm 1$ occur equally often).

As mentioned in the introduction, we obtain stronger results in this case
comparatively with the Riemann zeta function and other families of $L$-functions studied in this paper.
This is due to a careful study for the off-diagonal terms of moments of short Euler products of $L(\sigma,\chi_d)$ using
the following Lemma of [8] which is a consequence of the work of Graham and Ringrose [6] on bounds for character sums to smooth moduli
\begin{lem}[{Lemma 4.2 of [8]}]
Let $n$ be a positive number not a perfect square. Write $n=n_0\square$ where $n_0>1$ is square-free, and suppose that all prime factors of $n_0$ are below $y$. Let $l\geq 1$ be an integer and put $L=2^l$. Then
$$ \sums_{|d|\leq x}\chi_d(n)\ll x^{1-\frac{l}{8L}}\prod_{p|n}\left(1+\frac{1}{p^{1-l/8L}}\right)y^{1/3}n_0^{\frac{1}{7L}}d(n_0)^{l^2/L}.$$
\end{lem}
Using this lemma we prove the analogue of Proposition 4.1 for this family
\begin{pro} Let $x$ be large, and $y\leq (\log x)^2$ be a large real number.
Then uniformly for all real numbers $r$ such that
 $|r|y^{1-\sigma}\leq (1-\sigma)\log x\log_2 y/500$ we have
$$
\frac{\pi^2}{6x}\sums_{|d|\leq x} L(\sigma,\chi_d;y)^{r}= \ex\left(L(\sigma,\X_2;y)^{r}\right)+O\left(\exp\left(-\frac{\log x\log_2y}{40\log y}\right)\right).
$$
\end{pro}
\begin{proof}
Let $k$ be the smallest integer with $k\geq |r|$. We have that
\begin{equation}
\sums_{|d|\leq x} L(\sigma,\chi_d; y)^r=\sums_{|d|\leq x}\left(\sum_{n\in S(y)}\frac{\chi_d(n)}{n^{\sigma}}\right)^r=\sums_{|d|\leq x}\sum_{n\in S(y)}\frac{d_r(n)\chi_d(n)}{n^{\sigma}}.
\end{equation}
We begin by estimating the contribution of the diagonal terms $n=\square$ which give the main term of (5.1). Using that
$$ \sums_{|d|\leq x}\chi_d(n^2)=\sums_{\substack{|d|\leq x\\ (d,n)=1}}1= \frac{6}{\pi^2}x\prod_{p|n}\left(\frac{p}{p+1}\right)+O\left(x^{\frac{1}{2}+\epsilon}d(n)\right),$$
we deduce that the contribution of these terms is
\begin{equation}
\frac{6}{\pi^2}x\sum_{n\in S(y)}\frac{d_r(n^2)}{n^{2\sigma}}\prod_{p|n}\left(\frac{p}{p+1}\right)+ O\left(x^{\frac{1}{2}+\epsilon}\sum_{n\in S(y)}\frac{d_k(n^2)d(n)}{n^{2\sigma}}\right).
\end{equation}
Since $d_k(n^2)d(n)\leq d_{2k+2}(n)^2$, then the error term above is
$$ \ll x^{\frac{1}{2}+\epsilon}\sum_{n\in S(y)}\frac{d_{2k+2}(n)^2}{n^{2\sigma}}\ll
x^{\frac{1}{2}+\epsilon}\left(\sum_{n\in S(y)}\frac{d_{2k+2}(n)}{n^{\sigma}}\right)^2\ll
x^{\frac{1}{2}+\epsilon}\exp\left(O\left(\frac{ky^{1-\sigma}}{\log y}\right)\right),$$
which follows from (2.1). Moreover, we have that
\begin{equation}
\ex\left(L(\sigma,\X_2;y)^{r}\right)= \sum_{n\in S(y)}\frac{d_r(n^2)}{n^{2\sigma}}\prod_{p|n}\left(\frac{p}{p+1}\right).
\end{equation}
Now it remains to bound the contribution of the off-diagonal terms $n\neq \square$ to (5.1).
We use Lemma 5.1 to handle these terms. Write $n=n_1n_2^2n_3^2$ where $n_1,n_2$ are squarefree, with $(n_1,n_2)=1$, and $p|n_3\implies p|n_1n_2$: that is $n_1$ is the product of all primes dividing $n$ to an odd power (so $n_1>1$) and $n_2$ is the product of all primes dividing $n$ to an even power $\geq 2$. Observe that $\sums_{|d|\leq x}\chi_d(n)= \sums_{|d|\leq x}\chi_d(n_1n_2^2)$. Therefore these terms contribute
\begin{equation}
\sum_{\substack{n_1\in S(y)\\ n_1\neq 1}} \mu^2(n_1)\sum_{\substack{ n_2\in S(y)\\(n_1,n_2)=1}} \mu^2(n_2)\sums_{|d|\leq x}\chi_d(n_1n_2^2)\sum_{p|n_3\implies p|n_1n_2}\frac{d_r(n_1n_2^2n_3^2)}{(n_1n_2^2n_3^2)^{\sigma}}.
\end{equation}
Since $d_r(n)$ is a multiplicative function we obtain that
$$ \sum_{p|n_3\implies p|n_1n_2}\frac{d_r(n_1n_2^2n_3^2)}{(n_1n_2^2n_3^2)^{\sigma}}= \prod_{p|n_1}F(p)\prod_{p|n_2}H(p),$$
where
$$ F(p):=\sum_{a=0}^{\infty}\frac{d_r(p^{2a+1})}{p^{\sigma(2a+1)}}= \frac{1}{2}\left(\left(1-\frac{1}{p^{\sigma}}\right)^{-r}-\left(1+\frac{1}{p^{\sigma}}\right)^{-r}\right),$$
and
$$ H(p):=\sum_{a=0}^{\infty}\frac{d_r(p^{2a+2})}{p^{\sigma(2a+2)}}=
\frac{1}{2}\left(\left(1-\frac{1}{p^{\sigma}}\right)^{-r}+\left(1+\frac{1}{p^{\sigma}}\right)^{-r}\right)-1.$$
Using this and appealing to Lemma 5.1 we see that the sum (5.4) is
$$ \ll x^{1-\frac{l}{8L}}y^{1/3}\prod_{p\leq y}\left(1+2^{l^2/L}p^{1/(7L)}F(p)\left(1+\frac{1}{p^{1-l/(8L)}}\right)+ H(p)
\left(1+\frac{1}{p^{1-l/(8L)}}\right)\right),$$
for any positive integer $l\geq 1$ with $L=2^l$. We choose $l=[\log_2y/\log2]$ to get that
$$2^{l^2/L}\leq 2, p^{\frac{1}{7L}}\leq 2, \text{ and } 1+1/p^{1-\frac{l}{8L}}\leq 2, \text{ for all primes } p\leq y.$$
This implies that the sum (5.4) is bounded by
$$ x^{1-\log_2y/(10\log y)}\prod_{p\leq y}\left(1+8F(p)+ 2H(p)
\right).$$
Furthermore we know that
\begin{equation*}
F(p)\leq \left\{\begin{aligned} &2\frac{k}{p^{\sigma}} \ \ \ \ \text{ if } p>k^{1/\sigma},\\  &\left(1-\frac{1}{p^{\sigma}}\right)^{-k} \ \text{ if } p\leq k^{1/\sigma}, \end{aligned} \right.
\text{ and } H(p)\leq \left\{\begin{aligned} &2\frac{k^2}{p^{2\sigma}} \ \ \ \ \text{ if } p>k^{1/\sigma},\\  &\left(1-\frac{1}{p^{\sigma}}\right)^{-k} \ \text{ if } p\leq k^{1/\sigma}. \end{aligned}\right.
\end{equation*}
Then using these inequalities we obtain that the sum (5.4) is
$$
\ll x^{1-\frac{\log_2y}{10\log y}}\prod_{p\leq y}\left(20\left(1-\frac{1}{p^{\sigma}}\right)^{-k}\right)
\ll x\exp\left(-\frac{\log x\log_2y}{10\log y} + \frac{4 ky^{1-\sigma}}{(1-\sigma)\log y}\right),
$$
if $y\leq k^{1/\sigma}$, and is
\begin{equation*}
\begin{aligned}
&\ll x^{1-\frac{\log_2y}{10\log y}}\prod_{p\leq k^{1/\sigma}}\left(20\left(1-\frac{1}{p^{\sigma}}\right)^{-k}\right)\prod_{k^{1/\sigma}<p\leq y}\left(1+20\frac{k}{p^{\sigma}}\right)\\
&\ll x\exp\left(-\frac{\log x\log_2y}{10\log y}+ \frac{4k^{1/\sigma}}{(1-\sigma)\log k} + \frac{20ky^{1-\sigma}}{(1-\sigma)\log y}\right),
\end{aligned}
\end{equation*}
if  $y\geq k^{1/\sigma}$. In the two cases we get that the contribution of the off-diagonal terms is
 $\ll x\exp\left(-\frac{\log x\log_2y}{20\log y}\right)$, completing the proof.
\end{proof}
We now prove the analogue of Lemma 4.2
\begin{lem}
Let $x$ be large and $2\leq y\leq z$ be real numbers. Then for all positive integers $k$ such that $1\leq k\leq \log x/(6\log z)$ we have
$$ \sums_{|d|\leq x}\left|\sum_{y\leq p\leq z}\frac{\chi_d(p)}{p^{\sigma}}\right|^{2k}\ll x\frac{(2k)!}{2^k k!}\left(\sum_{y\leq p\leq z}\frac{1}{p^{2\sigma}}\right)^k +O\left(x^{3/4}\right).$$
\end{lem}

\begin{proof}
First we have that
$$ \sums_{|d|\leq x}\left|\sum_{y\leq p\leq z}\frac{\chi_d(p)}{p^{\sigma}}\right|^{2k}= \sums_{|d|\leq x}\sum_{y\leq p_1,...,p_{2k}\leq z}\frac{\chi_d(p_1...p_{2k})}{(p_1...p_{2k})^{\sigma}}.$$
The diagonal terms $p_1...p_{2k}=\square$ contribute
$$ \ll x\frac{(2k)!}{2^k k!}\left(\sum_{y\leq p\leq z}\frac{1}{p^{2\sigma}}\right)^k.$$
To handle the off-diagonal terms we use a result of  Granville and Soundararajan (Lemma 4.1 of [8]) which states that $ \sums_{|d|\leq x}\chi_d(n)\ll x^{1/2}n^{1/4}\log n,$
if $n\neq \square$. Thus if $p_1p_2...p_{2k}\neq \square$ and $p_i\leq z$ then
$ \sums_{|d|\leq x}\chi_d(p_1p_2...p_{2k})\ll x^{1/2}z^{k/2}\log x,$
which implies that the contribution of these terms is
$$ \ll x^{1/2}z^{k/2}\log x\left(\sum_{y\leq p\leq z}\frac{1}{p^{\sigma}}\right)^{2k}\ll x^{1/2}z^{(5/2-2\sigma)k}\log x\ll x^{3/4}.$$
\end{proof}
\begin{proof}[Proof of Theorem 1.6]
For simplicity write $\Phi_x(\tau)=\Phi_x^{\text{quad}}(\sigma,\tau)$. The proof is obtained by following the same lines as Theorems 1 and 4.5. Let $y=\log x\log_3 x$ and define $\Phi_x(\tau;y)$ to be the proportion of fundamental discriminants $d$ such that $|d|\leq x$ and  $\log|L(\sigma,\chi_d;y)|>\tau.$
Applying Proposition 5.3 gives that for all positive real numbers $r \leq (1-\sigma)(\log x\log_3 x)^{\sigma}/500,$ we have
$$
\int_{-\infty}^{\infty}re^{rt}\Phi_x(t;y)dt=\frac{1}{\left(\sums_{|d|\leq x}1\right)}\sums_{|d|\leq x}L(\sigma,\chi_d;y)^r=\ex\left(L(\sigma,\X_2;y)^{r}\right)+o(1).
$$
To estimate the moments of the random model we use Proposition 3.2 with $d=1$ and $Z(p)=X_2(p)$. Notice that the sequence $\{Z(p)\}_p$ satisfy hypothesis (ULD) with $Z$ being a random variable taking the values $1$ and $-1$ with equal probability $1/2$. Hence it follows that
$$\int_{-\infty}^{\infty}re^{rt}\Phi_x(t;y)dt=\exp\left(G_2(\sigma)
\frac{r^{1/\sigma}}{\log r}\left(1+O\left(\frac{1}{\log r}+\left(\frac{r^{1/\sigma}}{y}\right)^{2\sigma-1}\right)\right)\right),
$$
where $G_2(\sigma)=\int_{0}^{\infty}\log \cosh(u)u^{-1-\frac{1}{\sigma}}du.$
Therefore, using the saddle point method (as in the proof of Theorem 1.9) we get that
\begin{equation}
\Phi_x\left(\tau;y\right)=\exp\left(-A_2(\sigma)\tau^{\frac{1}{1-\sigma}}
(\log\tau)^{\frac{\sigma}{1-\sigma}}\left(1+O\left(\frac{1}{\sqrt{\log\tau}}+r(y,\tau)\right)\right)\right),
\end{equation}
 in a region $1\ll \tau\ll (\log x\log_3 x)^{1-\sigma}/\log_2 x$. Thus it only remains to construct a set $\mathcal{A}(x,\tau)$ which will play a similar role to that of $\mathcal{A}(T,\tau)$ in the proof of Theorem 1.1.
To this end we use the following zero-density result of Heath-Brown [11], which states that
for any $\delta>0$ we have
$\sums_{|d|\leq x} N(\sigma,T, \chi_d)\ll (xT)^{\delta}x^{3(1-\sigma)/(2-\sigma)}T^{(3-2\sigma)/(2-\sigma)}.$
Using this result along with Lemma 2.3 with $t=0$, $z=(\log x)^{3/(\sigma-1/2)}$
and $\sigma_0=\sigma/2+1/4>1/2$, give that
$$ \log L(\sigma,\chi_d)=\sum_{n=2}^z\frac{\Lambda(n)\chi_d(n)}{n^{\sigma}}+O\left(\frac{1}{(\log x)^{1/4}}\right),$$
for all fundamental discriminants $|d|\leq x$ except for a set $\mathcal{A}_0(x)$ of cardinality $\leq x^{1-a(\sigma)}$ for some constant $a(\sigma)>0$.
Now take $\epsilon=r(\log x\sqrt{\log_3x},\tau)$ and let $\mathcal{A}_1(x,\tau,\epsilon)$ be the set of fundamental discriminants $|d|\leq x$ such that $|\sum_{y\leq p\leq z}\chi_d(p)/p^{\sigma+it}|>\epsilon\tau.$  Then using Lemma 5.3 we see that
\begin{equation}
 \frac{1}{x}|\mathcal{A}_1(x,\tau,\epsilon)|\ll \left(\frac{3k}{(2\sigma-1)\epsilon^2\tau^2y^{2\sigma-1}\log y}\right)^k,
 \end{equation}
for all integers $k\leq (\sigma-\frac12)\log x/(18\log_2 x)$. Observing that $r(s,\tau)^2\tau^2s^{2\sigma-1}\log s = \tau^{\frac{1}{(1-\sigma)}}(\log s)^{\frac{\sigma}{(1-\sigma)}}$,  we deduce that $\epsilon^2\tau^2y^{2\sigma-1}\log y\geq \tau^{\frac{1}{(1-\sigma)}}(\log \tau)^{\frac{\sigma}{(1-\sigma)}}(\log_3 x)^{\sigma-\frac{1}{2}}$. If $\tau\leq (\log x)^{1-\sigma}/\log_2 x$ then we choose $k$ to be the largest integer below

\noindent $b_1\tau^{1/(1-\sigma)}
(\log\tau)^{\sigma/(1-\sigma)}$, for some suitably small constant $b_1>0$. In this case one can check that $|\mathcal{A}_1(x,\tau,\epsilon)|/x=o(\Phi_x(2\tau;y))$.

\noindent On the other hand, if $\tau\geq (\log x)^{1-\sigma}/\log_2 x$ we choose $k=[(\sigma-\frac12)\log x/(18\log_2x)]$. In this case it follows from (5.6) that $|\mathcal{A}_1(x,\tau,\epsilon)|/x\ll \exp(-b_2\log x\log_4 x/\log_2 x)$ for some constant $b_2>0$. This implies that $ |\mathcal{A}_1(x,\tau,\epsilon)|/x=o(\Phi_x(2\tau;y))$
 uniformly for
 $\tau \leq c_2(\sigma)(\log x\log_4 x)^{1-\sigma}/\log_2 x$, if $c_2(\sigma)$ is small enough.
Finally taking $\mathcal{A}(x,\tau)=\mathcal{A}_0(x)\cup\mathcal{A}_1(x,\tau,\epsilon)$,
we obtain that $|\mathcal{A}(x,\tau)|/x=o(\Phi_x(2\tau;y))$, and that $|\log L(\sigma,\chi_d)-\log L(\sigma,\chi_d;y)|<\epsilon_1\tau$,
for all fundamental discriminants $|d|\leq x$ with $d \notin \mathcal{A}(x,\tau)$, where $\epsilon_1=\epsilon+1/\log\tau$. This along with (5.5) completes the proof.
\end{proof}

\subsection{ 5.2. $\Omega$-results on GRH: proof of Theorem 1.8}

Let $s=\sigma+it$ where $1/2<\sigma<1$ and $t\in \mathbb{R}$.  Let $z\geq 2$ be a real number and define $P(z)=\prod_{p\leq z}p=e^{z+o(z)}$. For each prime $p\leq z$ let $\epsilon_p=\pm 1$, and denote by $\mathcal{P}_x(z,\{\epsilon_p\})$
 the set of primes $q\leq x$ such that $\left(\frac{p}{q}\right)=\epsilon_p$ for all primes $p\leq z$. Assuming GRH,  Granville and Soundararajan (see equation (9.1) of [8]) showed that
\begin{equation}
\sum_{q\in \mathcal{P}_x(z,\{\epsilon_p\})}\log q=\frac{x}{2^{\pi(z)}}+O\left(x^{\frac{1}{2}}\log^2(xP(z))\right).
\end{equation}
To prove Theorem 1.8, our strategy consists of computing the average of $\log |L(s,\chi_q)|$ over $q\in \mathcal{P}_x(z,\{\epsilon_p\})$, for some suitable set of signs $\{\epsilon_p\}_{p\leq z}$. We have

\begin{pro} Assume the GRH. Let $z$ be a real number with $2\leq z\leq (\log x)^2$. Then there exists a constant $B>0$ (which may depend only on $\sigma$) such that
$$ \sum_{q\in  \mathcal{P}_x(z,\{\epsilon_p\})} \log |L(s,\chi_q)|\log q= \frac{x}{2^{\pi(z)}}\sum_{p\leq z}\frac{\epsilon_p\cos(t\log p)}{p^{\sigma}}+ O\left(\frac{x}{2^{\pi(z)}}+x^{\frac{1}{2}}(\log x)^{B}\right).$$
\end{pro}
\begin{proof}
Since the GRH is assumed, then Lemma 2.3 gives that
$$ \log L(s,\chi_q)= \sum_{n=2}^{\log^{A} x} \frac{\Lambda(n)\left(\frac{n}{q}\right)}{n^{s}\log n}+ O\left(\frac{1}{\log^2 x}\right),$$
where $A= 4/(\sigma-\frac{1}{2})$. Using this estimate we obtain that
\begin{equation}
\sum_{q\in  \mathcal{P}_x(z,\{\epsilon_p\})} \log |L(s,\chi_q)|\log q=\text{Re}\sum_{q\in  \mathcal{P}_x(z,\{\epsilon_p\})} \sum_{n=2}^{\log^{A} x}\frac{\Lambda(n)\left(\frac{n}{q}\right)\log q }{n^{s}\log n} +E_5,
\end{equation}
where
$ E_5\ll x/2^{\pi(z)}+x^{\frac{1}{2}}\log^2 x,$
by (5.7).
To deal with the main term we define $\epsilon_l=\prod_{p|l}\epsilon_p$ and use the following identity
\begin{equation*}
\sum_{l|P(z)}\epsilon_l\left(\frac{l}{q}\right)=\prod_{p\leq z}\left(1+\epsilon_p\left(\frac{p}{q}\right)\right)=\left\{\begin{aligned}&2^{\pi(z)} \ \text{ if } q\in  \mathcal{P}_x(z,\{\epsilon_p\}),\\  &0 \ \text{ otherwise }. \end{aligned}\right.
\end{equation*}
This gives that
\begin{equation}
\sum_{q\in  \mathcal{P}_x(z,\{\epsilon_p\})}\sum_{n=2}^{\log^{A} x}\frac{\Lambda(n)\left(\frac{n}{q}\right)\log q }{n^{s}\log n}
= \frac{1}{2^{\pi(z)}}\sum_{l|P(z)}\epsilon_l\sum_{n=2}^{\log^{A} x}\frac{\Lambda(n) }{n^{s}\log n}\sum_{q\leq x}\left(\frac{nl}{q}\right)\log q.
\end{equation}
If $nl=\square$ then the inner sum above equals $\sum_{q\leq x}\log q+O\left(\sum_{p|ln}\log p\right)=x+O(x/\log^{4}x)$ by the prime number theorem. Moreover, since $n=p^{\alpha}$ and $l$ is square-free then $nl=\square$ if and only if $l=p$ and $\alpha=2m+1$ for some non-negative integer $m$. Hence the contribution of the diagonal terms $nl=\square$ to (5.9) equals
\begin{equation}
\frac{x}{2^{\pi(z)}}\sum_{p\leq z}\frac{\epsilon_p}{p^{s}}+O\left(\frac{x}{2^{\pi(z)}}\right),
\end{equation}
since $\sum_{p\leq z}\epsilon_pp^{-s}\ll z^{1-\sigma}\leq(\log x)^2$.
Now we bound the contribution of the off-diagonal terms $nl\neq \square$. In this case $\psi=\left(\frac{nl}{\cdot}\right)$ is a character of modulus $nl$ or $4nl$. Thus the inner sum over $q$ in (5.9) equals $\sum_{m\leq x}\psi(m)\Lambda(m)+O\left(x^{1/2}\right)\ll x^{\frac{1}{2}}\log^2(4nl)\ll x^{\frac{1}{2}}\log^4x,$ by GRH. This implies that the contribution of these terms to (5.9) is
$$ \ll x^{1/2}\log^4x\sum_{n=2}^{\log^{A} x}\frac{\Lambda(n) }{n^{\sigma}\log n}\ll x^{1/2}(\log x)^{4+A(1-\sigma)}.$$
This along with (5.8), (5.9) and (5.10) complete the proof.
\end{proof}
\begin{proof}[Proof of Theorem 1.8]
Let $2\leq z\leq (\log x)^2$ be a real number and for each $p\leq z$ take $\epsilon_p$ to be the sign of $\cos(t\log p)$. We shall only prove the first part of the Theorem since the second one can be deduced similarly (by taking $\epsilon_p$ to be minus the sign of $\cos(t\log p)$). Then Proposition 5.4 gives that
$$ \sum_{q\in  \mathcal{P}_x(z,\{\epsilon_p\})} \log |L(s,\chi_q)|\log q=\frac{x}{2^{\pi(z)}}\sum_{p\leq z}\frac{|\cos(t\log p)|}{p^{\sigma}}+ O\left(\frac{x}{2^{\pi(z)}}+x^{\frac{1}{2}}(\log x)^{B}\right),$$
for some $B>0$. Now $|\cos(t\log p)|\geq \cos(t\log p)^2=(1+\cos(2t\log p))/2$. Putting $s_0=\sigma+2it$ we deduce that
\begin{equation*}
\begin{aligned}\sum_{p\leq z}\frac{|\cos(t\log p)|}{p^{\sigma}}&\geq \frac{1}{2}\sum_{p\leq z}\left(\frac{1}{p^{\sigma}}+ \text{Re}\frac{1}{p^{s_0}}\right)\\
&= \frac{z^{1-\sigma}}{2(1-\sigma)\log z}+\text{Re}\frac{z^{1-s_0}}{2(1-s_0)\log z}+O\left(\frac{z^{1-\sigma}}{\log^2 z}\right),\\
\end{aligned}
\end{equation*}
by the prime number theorem. Now if $t=0$ the main term on the RHS of the last inequality equals $z^{1-\sigma}/((1-\sigma)\log z)$, otherwise we have
$$\frac{z^{1-\sigma}}{(1-\sigma)}+ \text{Re}\frac{z^{1-s_0}}{(1-s_0)}\geq z^{1-\sigma}\left(\frac{1}{1-\sigma}-\frac{1}{|1-s_0|}\right)\geq z^{1-\sigma}\frac{2 t^2}{(1-\sigma)\left((1-\sigma)^2+4 t^2\right)}.$$ This implies that
\begin{equation}
\sum_{q\in  \mathcal{P}_x(z,\{\epsilon_p\})} \log |L(s,\chi_q)|\log q\geq \alpha(s)\frac{xz^{1-\sigma}}{2^{\pi(z)}\log z}+ O\left(\frac{xz^{1-\sigma}}{2^{\pi(z)}\log^2 z}+x^{\frac{1}{2}}(\log x)^{B}\right),
\end{equation}
where $\alpha(s)=(1-\sigma)^{-1}$ if $t=0$, and $\alpha(s)= \alpha(\sigma)t^2/\left((1-\sigma)^2+4 t^2\right)$ otherwise.
Let $M_x$ be the number of primes $q\leq x$ such that $$\log |L(s,\chi_q)|\geq \alpha(s)\frac{z^{1-\sigma}}{\log z}\left(1-\frac{1}{\sqrt{\log z}}\right).$$
Since
$|\log L(s,\chi_q)|\leq \log x$ for all primes $q\leq x$ (which follows from GRH and Lemma 2.3), we deduce that
\begin{equation}
\sum_{q\in  \mathcal{P}_x(z,\{\epsilon_p\})} \log |L(s,\chi_q)|\log q\leq M_x\log^2x+ \alpha(s)\frac{z^{1-\sigma}}{\log z}\left(1-\frac{1}{\sqrt{\log z}}\right)\sum_{q\in  \mathcal{P}_x(z,\{\epsilon_p\})}\log q.
\end{equation}
Hence combining equations (5.7), (5.11) and (5.12) gives that
$$ M_x \geq \alpha(s)\frac{xz^{1-\sigma}}{2^{\pi(z)}\log^2x\log^{3/2}z} +O\left(\frac{xz^{1-\sigma}}{2^{\pi(z)}\log^2x\log^{2}z}+x^{1/2}\log^B x\right).$$
Choosing
$ z=\log x\log_2 x/(2\log 2)-\log x\sqrt{\log_2x},$
we conclude that there are $\gg x^{1/2}$ primes $q\leq x$ such that
$$ \log |L(s,\chi_q)|\geq (\beta(s)+o(1))(\log x)^{1-\sigma}(\log_2 x)^{-\sigma},$$ completing the proof.
\end{proof}

\section{The distribution of automorphic $L$-functions}

Let $1/2<\sigma<1$. In this section we study the distribution of the values $\log L(\sigma,f)$, for $f\in S_2^p(q)$, where $q$ is a large prime number. Given a sequence $(\alpha_f)_{f\in S_2^p(q)}$, its harmonic average is defined as the sum
$$ \sumh_{f\in S_2^p(q)}\alpha_f=\sum_{f\in S_2^p(q)}\frac{\alpha_f}{4\pi\langle f,f\rangle},$$
and if $S\subset S_2^p(q)$ then we will let $|S|_h$ denote the harmonic measure of $S$, that is
$$ |S|_h:=\sumh_{f\in S}1.$$
Such averaging is natural in view of the two following facts
$$ |S_2^p(q)|_h=1+O\left(\frac{\log q}{q^{3/2}}\right), \text{ and } \frac{1}{q(\log q)^3}\ll \omega_f\ll \frac{\log q}{q};$$
so that the harmonic weight $\omega_f$ is not far from the natural weight $1/|S_2^p(q)|$ (since $|S_2^p(q)|\asymp q$), and it defines asymptotically a probability measure on $S_2^p(q)$.

 Let $f\in S_2^p(q)$. For $2\leq y<q$ and $s\in \mathbb{C}$ define
$$ L(s,f;y)=\prod_{p\leq y}\left(1-\frac{e^{i\theta_f(p)}}{p^s}\right)^{-1}
\left(1-\frac{e^{-i\theta_f(p)}}{p^s}\right)^{-1}=\sum_{n\in S(y)}\frac{\lambda_f(n)}{n^s}.$$
We now describe the corresponding probabilistic model for this family. Let $\{\theta(p)\}_{p \text{ prime}}$ be independent random variables distributed on $[0,\pi]$ according to the Sato-Tate measure $\frac{2}{\pi}\sin^2t dt,$ and define $\X_3(p)=(X_3^1(p), X_3^2(p))$ where $X_3^1(p)=e^{i\theta(p)}$ and $X_3^2(p)=e^{-i\theta(p)}$. For $2\leq y$ and $s\in \Bbb{C}$ define the following random Euler product
$$ L(s,\X_3;y):=\prod_{p\leq y}\left(1-\frac{X_3^1(p)}{p^s}\right)^{-1}\left(1-\frac{X_3^2(p)}{p^s}\right)^{-1}.$$
We prove
\begin{pro} let $q$ be a large prime, and $y\leq (\log q)^{2}$ be a large real number. Then uniformly for all real numbers $r$ such that
$|r|y^{1-\sigma}\leq (1-\sigma)\log q/16$, we have that
$$\frac{1}{|S_2^p(q)|_h}\sumh_{f\in S_2^p(q)}L(\sigma,f;y)^r=\ex\left(L(\sigma,\X_3;y)^r\right)+O\left(\exp\left(-\frac{\log q}{4\log y}\right)\right).$$
\end{pro}

Let $r\in \Bbb{R}$. Then
$$ L(s,f;y)^r=\sum_{n\in S(y)}\frac{\lambda_{f,r}(n)}{n^s}=\prod_{p\leq y}\left(\sum_{a=0}^{\infty}\frac{\lambda_{f,r}(p^a)}{p^{as}}\right),$$
where $\lambda_{f,r}(n)$ is a multiplicative function. Our next lemma establishes a formula for $\lambda_{f,r}(p^a)$ in terms of $\lambda_f(p^{b})$ for $0\leq b\leq a$. Cogdell and Michel [2] achieved this in a more general context of compact groups, and for all symmetric powers of $f$ via representation theory. However in our specific case we can use a simple elementary approach  which avoids the representation theory language.
\begin{lem}
For any real number $r$ we have
\begin{equation}
\lambda_{f,r}(p^a)=\sum_{0\leq b\leq a}C_r(a,b)\lambda_f(p^b),
\end{equation}
where the coefficients $C_r(a,b)$ are defined by
$$ C_r(a,b):=\frac{1}{\pi}\sum_{\substack{ m,l\geq 0\\ m+l=a}} d_r(p^m)d_r(p^l)\int_0^{\pi}\cos((m-l)\theta)\sin\theta\sin((b+1)\theta)d\theta.$$ Moreover we have that
$|C_r(a,b)|\leq d_{2k}(p^a)$ for all $0\leq b\leq a$, where $k$ is the smallest integer with $k\geq |r|$.
\end{lem}
\begin{proof} We have that
\begin{equation*}
\begin{aligned}
\sum_{a=0}^{\infty}\frac{\lambda_{f,r}(p^a)}{p^{as}}&=\left(1-\frac{e^{i\theta_f(p)}}{p^s}\right)^{-r}
\left(1-\frac{e^{-i\theta_f(p)}}{p^s}\right)^{-r}\\
&= \sum_{m=0}^{\infty}\frac{d_r(p^m)e^{im\theta_f(p)}}{p^{ms}}
\sum_{l=0}^{\infty}\frac{d_r(p^l)e^{-il\theta_f(p)}}{p^{ls}},\\
\end{aligned}
\end{equation*}
which implies that
$$ \lambda_{f,r}(p^a)=\sum_{\substack{ m,l\geq 0\\ m+l=a}} d_r(p^m)d_r(p^l)e^{i(m-l)\theta_f(p)}=\frac{1}{2}\sum_{\substack{ m,l\geq 0\\ m+l=a}} d_r(p^m)d_r(p^l)\cos((m-l)\theta_f(p)).$$
Now recall that for any $b\geq 0$ we have that
$ \lambda_f(p^b)=\sin((b+1)\theta_f(p))/\sin\theta_f(p),$
and that the functions $\{S_n\}_{n\geq 0}$, defined by
$$S_n(\theta):=\frac{\sin((n+1)\theta)}{\sin\theta}$$
form an orthonormal basis of $L^2([0,\pi],\mu_{ST})$ where $\mu_{ST}$ is the Sato-Tate measure $\frac{2}{\pi}\sin^2\theta d\theta$ on $[0,\pi]$. These facts imply that
$$  Y_{r,a}(\theta):=\frac{1}{2}\sum_{\substack{m,l\geq 0\\ m+l=a}} d_r(p^m)d_r(p^l)\cos((m-l)\theta)= \sum_{0\leq b\leq a}C_r(a,b)S_b(\theta),$$
since $Y_{r,a}(\theta)$ is a trigonometric polynomial of degree $\leq a$, and the coefficients $C_r(a,b)$ are defined by
\begin{equation*}
\begin{aligned}
C_r(a,b)&=\frac{1}{\pi}\int_0^{\pi}\sum_{\substack{ m,l\geq 0\\ m+l=a}} d_r(p^m)d_r(p^l)\cos((m-l)\theta)S_b(\theta)\sin^2\theta d\theta\\
&=\frac{1}{\pi}\sum_{\substack{ m,l\geq 0\\ m+l=a}} d_r(p^m)d_r(p^l)\int_0^{\pi}\cos((m-l)\theta)\sin\theta\sin((b+1)\theta)d\theta.\\
\end{aligned}
\end{equation*}
Furthermore, noting that $\lambda_{f,r}(p^a)=Y_{r,a}(\theta_f(p))$ and $\lambda_f(p^b)=S_b(\theta_f(p))$, gives (6.1). Finally the last estimate follows from the fact that
$$ |C_r(a,b)|\leq \sum_{\substack{ m,l\geq 0\\ m+l=a}} d_k(p^m)d_k(p^l)\leq d_{2k}(p^a).$$
\end{proof}
For any positive integer $n=p_1^{a_1}\cdots p_j^{a_j}$, define
$ \lambda^{\text{rand}}_{r}(n):=\prod_{i=1}^jC_r(a_i,0).$
Then using the Petersson trace formula we prove the following lemma
\begin{lem}
For all positive integers $n$ with $(n,q)=1$, and all real numbers $r$ we have
$$ \frac{1}{|S_2^p(q)|_h}\sumh_{f\in S_2^p(q)}\lambda_{f,r}(n)=\lambda^{\text{rand}}_{r}(n)+O\left(q^{-3/2}n^{1/2}\log(qn)d_{4k}(n)\right),$$
where $k$ is the smallest integer with $k\geq |r|$.
\end{lem}

\begin{proof}
Write $n=p_1^{a_1}\cdots p_j^{a_j}$. Then by Lemma 6.2 we have that
\begin{equation}
\begin{aligned}
\frac{1}{|S_2^p(q)|_h}\sumh_{f\in S_2^p(q)}\lambda_{f,r}(n)&=
\frac{1}{|S_2^p(q)|_h}\sumh_{f\in S_2^p(q)}\prod_{i=1}^j\lambda_{f,r}(p_i^{a_i})\\
&=\frac{1}{|S_2^p(q)|_h}\sumh_{f\in S_2^p(q)}\prod_{i=1}^j\left(\sum_{0\leq b_i\leq a_i}C_r(a_i,b_i)\lambda_f(p_i^{b_i})\right)\\
&= \sum_{0\leq b_1\leq a_1}\cdots \sum_{0\leq b_j\leq a_j} \prod_{i=1}^jC_r(a_i,b_i)\frac{1}{|S_2^p(q)|_h}\sumh_{f\in S_2^p(q)}\lambda_f(p_1^{b_1}\cdots p_j^{b_j}).\\
\end{aligned}
\end{equation}
Now applying the Petersson trace formula (see [12] and [2])
\begin{equation}
 \frac{1}{|S_2^p(q)|_h}\sumh_{f\in S_2^p(q)}\lambda_f(m)=\delta_{m,1}+O\left(\frac{\log(qm)m^{1/2}}{q^{3/2}}\right),
\end{equation}
to the inner sum on the RHS of (6.2) gives that
$$ \frac{1}{|S_2^p(q)|_h}\sumh_{f\in S_2^p(q)}\lambda_{f,r}(n)=\lambda^{\text{rand}}_{r}(n)+ E_6,$$
where
\begin{equation*}
\begin{aligned}
E_6 &\ll \frac{\log(qn)}{q^{3/2}}\sum_{0\leq b_1\leq a_1}\cdots \sum_{0\leq b_j\leq a_j} \prod_{i=1}^j|C_r(a_i,b_i)|(p_1^{b_1}\cdots p_j^{b_j})^{\frac12}\\
&\ll q^{-3/2}n^{1/2}\log(qn)d_{2k}(n)d(n)\ll q^{-3/2}n^{1/2}\log(qn)d_{4k}(n),
\end{aligned}
\end{equation*}
which follows from Lemma 6.2 along with the fact that $\prod_{i=1}^j(a_i+1)=d(n).$
\end{proof}
\begin{proof}[Proof of Proposition 6.1]
Let $k$ be the smallest integer with $k\geq |r|$. Then by Deligne's bound $|\lambda_f(n)|\leq d(n)$ and Lemma 6.2 we have that $$|\lambda_{f,r}(p^a)|\leq \sum_{0\leq b\leq a} d_{2k}(p^a)d(p^b)=d_{2k}(p^a)d_3(p^a)\leq d_{6k}(p^a),$$
and so $|\lambda_{f,r}(n)|\leq d_{6k}(n)$ for all positive integers $n$ by multiplicativity. Therefore Lemma 2.5 implies that
$$ L(\sigma,f;y)^r=\sum_{\substack{ n\leq q^2\\ n\in S(y)}} \frac{\lambda_{f,r}(n)}{n^{\sigma}}+O\left(\exp\left(-\frac{2\log q}{\log y}+ \log_2 q+ \frac{6eky^{1-\sigma}}{(1-\sigma)\log y}(1+o(1))\right)\right),$$
and observe that the error term above is
$ \ll \exp\left(-\frac{\log q}{4\log y}\right),$ by our hypothesis on $r$.
Furthermore Lemma 6.3 gives that
$$ \frac{1}{|S_2^p(q)|_h}\sumh_{f\in S_2^p(q)}\sum_{\substack{n\leq q^2\\ n\in S(y)}} \frac{\lambda_{f,r}(n)}{n^{\sigma}}= \sum_{\substack{ n\leq q^2\\ n\in S(y)}} \frac{\lambda^{\text{rand}}_{r}(n)}{n^{\sigma}}+O\left(q^{-3/2}\log q\sum_{\substack{ n\leq q^2\\ n\in S(y)}} \frac{d_{4k}(n)}{n^{\sigma-1/2}}\right).$$
Now the error term above is
$$ \ll \frac{\log q}{q^{1/2}}\sum_{n\in S(y)}\frac{d_{4k}(n)}{n^{\sigma}}\ll\frac{\log q}{q^{1/2}}\exp\left(5\frac{ky^{1-\sigma}}{(1-\sigma)\log y}\right)\ll q^{-1/8},$$
which follows from (2.1) and our hypothesis on $k$.
Moreover, notice that
\begin{equation*}
\begin{aligned}
\ex\left(L(s,\X_3;y)^r\right)&=\prod_{p\leq y}\ex\left(\left(1-\frac{X_3^1(p)}{p^{\sigma}}\right)^{-r}\left(1-\frac{X_3^2(p)}{p^{\sigma}}\right)^{-r}\right)=\prod_{p\leq y}\left(\sum_{a=0}^{\infty}\frac{C_r(a,0)}{p^{as}}\right)\\
&=\sum_{n\in S(y)}\frac{\lambda^{\text{rand}}_r(n)}{n^s}.\\
\end{aligned}
\end{equation*}
Finally, since $|\lambda^{\text{rand}}_{r}(n)|\leq d_{2k}(n)$ by Lemma 6.2, then applying Lemma 2.5 gives that
$$  \sum_{\substack{n\leq q^2\\ n\in S(y)}} \frac{\lambda^{\text{rand}}_{r}(n)}{n^{\sigma}}= \ex\left(L(\sigma,\X_3;y)^r\right)+ O\left(\exp\left(-\frac{\log q}{4\log y}\right)\right),$$ completing the proof.
\end{proof}
In order to prove Theorem 1.7, we need a large sieve type inequality for the Fourier coefficients $\lambda_f(p)$, analogous to Lemmas 4.4 and 5.3. In this case the argument is different since $\lambda_f(n)$ is not completely  multiplicative. A crucial role will be played by the Hecke relations:
\begin{equation}
 \lambda_f(m)\lambda_f(n)=\sum_{d|(m,n)}\lambda_f\left(\frac{mn}{d^2}\right).
\end{equation}
To understand the combinatorics of these relations, Rudnick and Soundararajan [22] introduced a ring $\mathcal{H}$ generated over the integers by symbols $x(n)(n\in\mathbb{N})$ which satisfy the following
$$ x(1)=1, \text{ and } x(m)x(n)=\sum_{d|(m,n)} x\left(\frac{mn}{d^2}\right).$$
Therefore, using the Hecke relations (6.4) we may write
$$ x(n_1)\cdots x(n_r)=\sum_{m|\prod_{j=1}^rn_j}b_m(n_1,\dots,n_r)x(m),$$
for some integers $b_m(n_1,\dots,n_r)$.
Rudnick and Soundararajan explored some properties of these coefficients, showing that $b_m(n_1,\dots,n_r)$ is always non-negative and is symmetric in the variables $n_1,\dots, n_r$, and finally that $b_m(n_1,\dots, n_r)\leq d(n_1)\cdots d(n_r)$. The case $m=1$ is of special importance to us. Here Rudnick and Soundararajan noticed that $b_1$ satisfies a multiplicative property:
\begin{equation}
b_1(m_1n_1,\dots,m_rn_r)=b_1(m_1,\dots,m_r)b_1(n_1,\dots,n_r),
\end{equation}
if $(\prod_{i=1}^rm_i,\prod_{i=1}^rn_i)=1$; and that for a prime $p$ we have $b_1(p^{a_1},\dots,p^{a_r})=0$ if $a_1+\cdots + a_r$ is odd.
We prove the following lemma
\begin{lem} Let $k$ be a positive integer, $2\leq y\leq z$ be real numbers and $\{\alpha(p)\}_{p \text{ prime }}$ be a sequence of real numbers. Then
$$ \sum_{y\leq p_1,\dots,p_{2k}\leq z}\alpha(p_1)\cdots \alpha(p_{2k})b_1(p_1,\dots, p_{2k})\leq 2^{2k}\frac{(2k)!}{k!}\left(\sum_{y\leq p\leq z}\alpha(p)^2\right)^k.
$$
\end{lem}
\begin{proof}
Let us define
$$B_{2k}(n)=\sum_{\substack{p_1,\dots,p_{2k}\\ p_1\cdots p_{2k}=n}}b_1(p_1,\dots,p_{2k}),$$
First one can see that $B_{2k}(n)=0$ unless $\Omega(n)= 2k$ (where $\Omega(n)$ is the total number of prime factors of $n$ counted with multiplicities) and $n$ is a perfect square. This follows from the multiplicative property (6.5) along with the fact that $b_1(p,\dots,p,1,\dots,1)=0$ if the number of occurrence of the prime $p$ is odd. Let $n$ be a square with $\Omega(n)=2k$. Write $n=q_1^{2a_1}\cdots q_j^{2a_j}$ where $q_1<q_2<\dots<q_j$ are distinct primes,  and the $a_i$ are positive integers with $a_1+\cdots +a_j=k$.
Using that $b_1(p_1,\dots,p_{2k})\leq d(p_1)\cdots d(p_{2k})=2^{2k}$ we obtain that
$$ B_{2k}(n)\leq 2^{2k} \sum_{\substack{p_1,\dots,p_{2k}\\ p_1\cdots p_{2k}=n}} 1 = 2^{2k}\binom{2k}{2a_1,...,2a_j}\leq 2^{2k}\frac{(2k)!}{k!}\binom{k}{a_1,...,a_j}.
$$
Thus we get that
\begin{equation*}
\begin{aligned}
&\sum_{y\leq p_1,\dots,p_{2k}\leq z}\alpha(p_1)\cdots \alpha(p_{2k})b_1(p_1,\dots, p_{2k})\\
&= \sum_{y\leq q_1<\dots <q_{j}\leq z}\sum_{\substack{a_1,\dots,a_j\geq 1\\
a_1+\dots +a_j=k}}\alpha(q_1)^{2a_1}\cdots \alpha(q_j)^{2a_j}B_{2k}(q_1^{2a_1}\cdots q_j^{2a_j})\\
&\leq 2^{2k}\frac{(2k)!}{k!} \sum_{y\leq q_1<\dots <q_{j}\leq z}\sum_{\substack{a_1,\dots,a_j\geq 1\\
a_1+\dots +a_j=k}}\alpha(q_1)^{2a_1}\cdots \alpha(q_j)^{2a_j}\binom{k}{a_1,...,a_j}\\
&= 2^{2k}\frac{(2k)!}{k!}\left(\sum_{y\leq p\leq z}\alpha(p)^2\right)^k.
\end{aligned}
\end{equation*}
\end{proof}
Combining this result with the Petersson trace formula (6.3) we deduce the following lemma
\begin{lem}
Let $\sigma>1/2$. Let $q$ be a large prime and $1\ll y\leq z$ be real numbers. Then for all positive integers $k$ such that $1\leq k\leq \log q/(2\log z)$ we have
\begin{equation}
\frac{1}{|S_2^p(q)|_h}\sumh_{f\in S_2^p(q)}\left(\sum_{y\leq p\leq z}\frac{\lambda_f(p)}{p^{\sigma}}\right)^{2k}\ll 2^{2k}\frac{(2k)!}{k!}\left(\sum_{y\leq p\leq z}\frac{1}{p^{2\sigma}}\right)^k +q^{-1/2}.
\end{equation}
\end{lem}
\begin{proof} Expanding the LHS of (6.6) we obtain
$$\sum_{y\leq p_1,\dots, p_{2k}\leq z}\frac{1}{(p_1\cdots p_{2k})^{\sigma}}\sum_{m|(p_1\cdots
p_{2k})}b_m(p_1,\dots,p_{2k})\frac{1}{|S_2^p(q)|_h}\sumh_{f\in S_2^p(q)}\lambda_f(m).$$
Using the Petersson trace formula (6.3) along with the facts that $b_m(p_1,\dots, p_{2k})\leq 2^{2k}$ and $p_1\cdots p_{2k}\leq q$, we deduce that the contribution of the terms $m\neq 1$ is
$$ \ll \frac{\log q}{q}\left(4\sum_{y\leq p\leq z}\frac{1}{p^{\sigma}}\right)^{2k}\ll \frac{z^{2k(1-\sigma)}\log q}{q}\ll q^{-1/2}, $$
which follows from our hypothesis on $k$. On the other hand using Lemma 6.4 we deduce that the contribution of the term $m=1$ is
$$ \ll \sum_{y\leq p_1,\dots, p_{2k}\leq z}\frac{b_1(p_1,\dots,p_{2k})}{(p_1\cdots p_{2k})^{\sigma}}\ll 2^{2k}\frac{(2k)!}{k!}\left(\sum_{y\leq p\leq z}\frac{1}{p^{2\sigma}}\right)^k,$$
completing the proof.
\end{proof}

\begin{proof}[Proof of Theorem 1.7]

We follow exactly the proof of Theorem 4.5: first, we replace Proposition 4.3 with Proposition 6.1, then we estimate the moments of $L(\sigma,\X_3;y)$ using Proposition 3.2. Further, we replace Lemma 4.4 with Lemma 6.5 and Lemma 2.3 with Lemma 2.4. The only ingredient which remains is zeros-density estimates. This has been achieved by Kowalski and Michel in [13]. Let $N(f,\alpha,T)$ denote the number  of zeros of $L(s,f)$ such that Re$(s)\geq \alpha$ and $|\text{Im}(s)|\leq T$. Then Theorem 4 of [13] states that for a large prime $q$ and $T>1$, there exists an absolute constant $A>0$ such that for any $\alpha\geq 1/2+(\log q)^{-1}$, and for any $c$, $0<c<1/4$, one has
$$ \sum_{f\in S_2^p(q)}N(f,\alpha,T)\ll  T^{A}q^{1-c(\alpha-1/2)}\log q.$$
\end{proof}

\section{Large moments of the Riemann zeta function on Re$(s)=\sigma$}

\subsection{Estimating complex moments of $\zeta(\sigma+it)$ under RH: Proof of Theorem 1.2}

The key ingredient in the proof of Theorem 1.2 is to show that, under RH, one can approximate complex powers of $\zeta(\sigma+it)$ by very short Dirichlet polynomials. Specifically we prove
\begin{pro}
Assume the RH. Let $t$ be a real number with $|t|$ large, and let $|t|^{1/8}\leq X\leq |t|$. Let $0<\epsilon<1/2$. Then there exists $A>0$ such that uniformly for  $1/2+A/\log_2 |t|\leq \sigma\leq 1-\epsilon$ we have
$$ \zeta(\sigma+it)^z=\sum_{n=1}^{\infty}\frac{d_z(n)}{n^{\sigma+it}}e^{-n/X}+O\left(\exp\left(-\frac{\log |t|}{20\log_2 |t|}\right)\right),$$
for all complex numbers $z$ in the region $|z|\leq b(A)(\log |t|)^{2\sigma-1}$, where $b(A)$ is some positive constant.
\end{pro}
\begin{proof}
Without loss of generality assume that $t>0$ is large. Let $c=1-\sigma+1/\log X$. Then
$$ \frac{1}{2\pi i} \int_{c-i\infty}^{c+i\infty}\zeta(s+\sigma+it)^z\Gamma(s)X^sds= \sum_{n=1}^{\infty}\frac{d_z(n)}{n^{\sigma+it}}e^{-n/X}.$$
Before moving the contour to the left we will bound the contribution of the parts from $c+i\log^3t$ to $c+i\infty$ and from $c-i\infty$ to $c-i\log^3t$ using Stirling's formula. Since
$$ |\log \zeta(1+1/(\log X)+it)|\leq \log\zeta(1+1/(\log X))\ll \log_2 X,$$
 then we get that
\begin{equation*}
\begin{aligned}
&\frac{1}{2\pi i}\left(\int_{c-i\infty}^{c-i\log^3t}+\int_{c+i\log^3t}^{c+i\infty}\right)\zeta(s+\sigma+it)^z\Gamma(s)X^s ds\\
&\ll X^{1-\sigma}e^{O(|z|\log_2 X)}\int_{\log^3t}^{\infty}e^{-\frac{\pi}{3}u}du\ll \frac{1}{t}.\\
\end{aligned}
\end{equation*}
Now we shift the line of integration to the path $\mathcal{C}$ joining $c-i\log^3t$, $-\eta-i\log^3t$, $-\eta+i\log^3t$ and $c+i\log^3t$, where $\eta=1/\log_2 t$. Since we are assuming the RH we only encounter a simple pole at $s=0$ which leaves the residue $\zeta(\sigma+it)^z$. Moreover, if $A$ is large enough, the RH implies that uniformly for $1/2+(A-1)/\log_2 t\leq \alpha<1-\epsilon$, we have that  (see equation 14.14.5 of [27])
\begin{equation}
|\log\zeta(\alpha+it)|\leq c(A)\frac{(\log t)^{2-2\alpha}}{\log_2 t},
\end{equation}
for some $c(A)>0$. Then using this bound along with Stirling's formula, we deduce that uniformly for $1/2+A/\log_2 t\leq \sigma<1-\epsilon$ we have
\begin{equation*}
\begin{aligned}
&\frac{1}{2\pi i}\left(\int_{c-i\log^3t}^{-\eta-i\log^3t}+\int_{-\eta-i\log^3t}^{-\eta+i\log^3t}+\int_{-\eta+i\log^3t}^{c+i\log^3t}\right)\zeta(s+\sigma+it)^z\Gamma(s)X^s ds\\
&\ll \exp\left(-\frac{\pi}{3}\log^3t +O(|z|\log t)\right)X^{1-\sigma}+\frac{1}{\eta}(\log t)^3 X^{-\eta}\exp\left(c(A)|z|\frac{(\log t)^{2-2\sigma+2\eta}}{\log_2 t}\right)\\
&\ll \exp\left(-\frac{\log t}{20\log_2 t}\right),
\end{aligned}
\end{equation*}
if $|z|\leq b(A)(\log |t|)^{2\sigma-1}$, and $b(A)=1/(100c(A))$.
\end{proof}

\begin{proof}[Proof of Theorem 1.2] Let $X=2T^{1/8}$, and denote by $k$ the smallest integer with $k\geq \max(|z_1|,|z_2|)$. Then Proposition 7.1 implies that for all real numbers $t$ with $|t|\in [T,2T]$ we have
$$ \zeta(\sigma+it)^{z}=\sum_{n=1}^{\infty}\frac{d_z(n)}{n^{\sigma+it}}e^{-n/X}+O\left(\exp\left(-\frac{\log T}{20\log_2 T}\right)\right),$$
uniformly for $1/2+A/\log_2 T\leq \sigma\leq 1-\epsilon$, and all complex numbers $z$ in the region $|z|\leq b(A)(\log T)^{2\sigma-1}$. Hence we obtain that
$$
\frac{1}{T}\int_T^{2T}\zeta(\sigma+it)^{z_1}\overline{\zeta(\sigma+it)}^{z_2}dt
=\sum_{m,n\geq 1}\frac{d_{z_1}(n)d_{z_2}(m)}{(mn)^{\sigma}}e^{-(m+n)/X}\frac{1}{T}\int_T^{2T}\left(\frac{m}{n}\right)^{it}dt+E_7,
$$
where $$E_7\ll \exp\left(-\frac{\log T}{20\log_2 T}\right)\max_{t\in [T,2T]}|\zeta(\sigma+it)^{z_1}\overline{\zeta(\sigma+it)}^{z_2}|\ll \exp\left(-\frac{\log T}{50\log_2 T}\right),
$$
which follows from (7.1) (recall that we chose $b(A)=1/(100c(A))$ in the proof of Proposition 7.1). First we estimate the contribution of the diagonal terms $m=n$. Let $0<\alpha\leq 1$ be a real number to be chosen later. Then using that $1-e^{-t}\leq 2t^{\alpha}$ for all $t>0$, we deduce that the contribution of these terms equals
$$ \sum_{n=1}^{\infty}\frac{d_{z_1}(n)d_{z_2}(n)}{n^{2\sigma}}e^{-2n/X}=\sum_{n=1}^{\infty}\frac{d_{z_1}(n)d_{z_2}(n)}{n^{2\sigma}}+ O\left(X^{-\alpha}\sum_{n=1}^{\infty}\frac{d_k(n)^2}{n^{2\sigma-\alpha}}\right).$$
Let $\alpha = \min(\sigma-1/2, 1-\sigma)$. Then using Lemma 2.1 we deduce that the error term above is
$$ \ll\exp\left(-\frac{\alpha}{8}\log T+ O\left(\log_2 T\frac{k^{1/(\sigma-\alpha/2)}}{\log k}\right)\right)\ll \exp\left(-\frac{\alpha}{10}\log T\right),$$
using our hypothesis on $k$ since the maximum of $(2\sigma-1)/(\sigma-\alpha/2)$ over the interval $[1/2+A/\log_2 T, 1-\epsilon]$ is $<1$.

Next we bound the contribution of the off-diagonal terms. First if $m\neq n$ and $\max(m,n)<\sqrt{T}$ then $\int_{T}^{2T}(m/n)^{it}dt\ll \sqrt{T}$, which implies that the contribution of these terms is
$$ \ll \frac{1}{\sqrt{T}}\left(\sum_{n\leq \sqrt{T}}\frac{d_k(n)}{n^{\sigma}}e^{-n/X}\right)^2\ll \frac{1}{\sqrt{T}}
X^{2-2\sigma}(\log 3 X)^{2k+2}\ll T^{-1/4},$$
using (2.3) and our hypothesis on $k$ and $X$.
Now we bound the contribution of the remaining terms $m\neq n$ such that $\max(m,n)>\sqrt{T}$. Let $\beta=1-\sigma$, then by (2.3) these terms contribute
\begin{equation*}
\begin{aligned} \ll \sum_{n>\sqrt{T}}\sum_{m\geq 1}\frac{d_k(n)d_k(m)}{(mn)^{\sigma}}e^{-(m+n)/X}&\ll T^{-\frac{\beta}{2}}X^{1-\sigma}(\log 3 X)^{k+1}\left(\sum_{n\geq 1}\frac{d_k(n)}{n^{\sigma-\beta}}e^{-n/X}\right)\\
&\ll T^{-\frac{(1-\sigma)}{2}}X^{3(1-\sigma)}(\log 3 X)^{2k+2}\ll T^{-(1-\sigma)/10},\\
\end{aligned}
\end{equation*}
proving the Theorem.
\end{proof}

\subsection{7.2 Exploring the range of validity for the asymptotic formula (1.7): proof of Theorems 3a and 3b}

In order to prove Theorems 3a and 3b, the first step consists in controlling the size of the derivative of $\zeta(s)$ on the line Re$(s)=\sigma$. For $\sigma=1/2$, Farmer, Gonek and Hughes [5] achieved this using the symmetry of the functional equation of $\zeta(s)$ about the line Re$(s)=1/2$. However we can not use such a tool for $\sigma>1/2$ since there is no symmetry in this case. Instead we use a Phragmen-Lindel\"of type argument that gives us a weaker bound, but will still  be sufficient for our purposes.
\begin{lem}
Let $1/2<\sigma<1$ and suppose that the asymptotic relation (1.7) holds for all integers $k\leq (\log T)^{\delta}$ for some $0<\delta\leq\sigma$. Then
$$m_T:=\max_{t\in [T,2T]}|\zeta(\sigma+it)|\leq \exp\left(O\left((\log T)^{1-\delta}\right)\right) \text{ and } |\zeta'(\sigma+it)|\ll_{\epsilon}t^{\epsilon},$$
for any $\epsilon>0$.
\end{lem}
\begin{proof}
Let $s=\sigma+it$. By Cauchy's theorem we have that
\begin{equation}
\zeta'(s)=\frac{1}{2\pi i}\oint_{|z-s|=r}\frac{\zeta(z)}{(z-s)^2}dz.
\end{equation}
Taking $r=\sigma/2-1/4>0$, and inserting the standard bound  $|\zeta(z)|\ll t^{1/6}$ for Re$(z)\geq 1/2$ into (7.2) we find that
$ |\zeta'(s)|\ll t^{1/6}.$ Let $t_0\in [T,2T]$ be such that $m_T=|\zeta(\sigma+it_0)|.$ Then for any $t\in [T,2T]$ with $|t-t_0|<T^{-1/6}$ we have that
$$ |\zeta(\sigma+it)-\zeta(\sigma+it_0)|\leq |t-t_0|\max_{x\in [T,2T]}|\zeta'(\sigma+ix)|\ll 1,$$
which in view of (1.2) implies that $|\zeta(\sigma+it)|\geq m_T/2$ for all $t$ with $|t-t_0|<T^{-1/6}$. Furthermore for all integers $k\leq (\log T)^{\delta}$ we have by (1.7) that
$$ \sum_{n=1}^{\infty}\frac{d_k^2(n)}{n^{2\sigma}}\gg \frac{1}{T}\int_{t_0-T^{-1/6}}^{t_0+T^{-1/6}}|\zeta(\sigma+it)|^{2k}dt\gg T^{-2}\left(\frac{m_T}{2}\right)^{2k}.$$
Now using Corollary 3.4, we find that
$$ m_T\ll \exp\left(\frac{\log T}{k}+O\left(\frac{k^{1/\sigma-1}}{\log k}\right)\right),$$
which gives the desired bound on $m_T$, upon taking $k=[\log^{\delta}T]$. In particular this shows that $|\zeta(\sigma+it)|\ll _{\epsilon} t^{\epsilon}$. Therefore the Phragmen-Lindel\"of principle implies that for any $\epsilon>0$ there exists a constant $c(\epsilon)>0$ such that one has
$$ |\zeta(\alpha+it)|\ll_{\epsilon}t^{\epsilon} \text{ for } \sigma-c(\epsilon)<\alpha.$$
Inserting this in (7.2) and taking $r=c(\epsilon)/2$ gives the desired bound on $|\zeta'(s)|$.
\end{proof}
\begin{proof}[Proof of Theorem 1.4]
The first implication follows from Lemma 7.2. Furthermore the proof of the second one follows exactly along the lines of the proof of Theorem 1.2, since we don't need the assumption of RH if $z=k\in \Bbb{N}$ in Proposition 7.1.

\end{proof}

\begin{proof}[Proof of Theorem 1.3]
 Let $\epsilon>0$ be a suitably small constant. Then following the proof of Lemma 7.2 and using the bound of $|\zeta'(s)|$ on Re$(s)=\sigma$ we can show that for any $t$ such that $|t-t_0|<T^{-\epsilon}$ we have that $|\zeta(\sigma+it)|\geq m_T/2$. Let $l$ be a large real number for which (1.7) holds. Then one has
$$ \sum_{n=1}^{\infty}\frac{d_l^2(n)}{n^{2\sigma}}\gg \frac{1}{T}\int_{t_0-T^{-\epsilon}}^{t_0+T^{-\epsilon}}|\zeta(\sigma+it)|^{2l}dt\gg T^{-(1+\epsilon)}\left(\frac{m_T}{2}\right)^{2l}.$$
Therefore Corollary 3.4 gives that
$$ m_T\ll_{\epsilon} \exp\left(\frac{(1+\epsilon)\log T}{2l}+G_1(\sigma)\frac{(2l)^{1/\sigma-1}}{\log l}(1+o(1))\right).$$
Setting $l= b(\log T\log_2 T)^{\sigma}$, we find that
\begin{equation}
 m_T\ll_{\epsilon} \exp\left(\left(\frac{1}{2b}+\frac{G_1(\sigma)(2b)^{\frac{1}{\sigma}-1}}
{\sigma}+O(\epsilon)\right)\frac{(\log T)^{1-\sigma}}{(\log_2 T)^{\sigma}}\right).
\end{equation}
Moreover a simple calculation shows that the function  $f(x)=(2x)^{-1}+G_1(\sigma)(2x)^{\frac{1}{\sigma}-1}/\sigma$ is minimized when $x_0=\frac12(\sigma^2/(G_1(\sigma)(1-\sigma)))^{\sigma}$, and its minimum equals $C(\sigma)=G_1(\sigma)^{\sigma}\sigma^{-2\sigma}(1-\sigma)^{\sigma-1}$. Furthermore, if $k$ is a large real number for which (1.7) holds then
$$\sum_{n=1}^{\infty}\frac{d_k^2(n)}{n^{2\sigma}}\ll \frac{1}{T}\int_{T}^{2T}|\zeta(\sigma+it)|^{2k}dt\ll m_T^{2k},$$
which in view of Corollary 3.4 gives that
\begin{equation}
m_T\gg \exp\left(G_1(\sigma)\frac{(2k)^{1/\sigma-1}}{\log k}(1+o(1))\right).
\end{equation}
Hence if (1.7) holds for $k=c(\log T\log_2 T)^{\sigma}$ then
\begin{equation}
 m_T\gg \exp\left(\left(\frac{G_1(\sigma)(2c)^{\frac{1}{\sigma}-1}}
{\sigma}+o(1)\right)\frac{(\log T)^{1-\sigma}}{(\log_2 T)^{\sigma}}\right),
\end{equation}
which gives a contradiction to (7.3) if $c>\frac12(B(\sigma))^{\sigma}$ and $\epsilon$ is sufficiently small.
\end{proof}

\subsection{7.3. Lower bounds for the moments: Proof of Theorem 1.5}

We follow the approach of Rudnick and Soundararajan [21]. For a real number $x$ and a positive integer $k$ we define $d_k(n;x)$ to be the number of ways of writing $n$ as $a_1\cdots a_k$ with $a_i$ being positive integers such that $a_i\leq x$. Note that $d_k(n;x)\leq d_k(n)$ with equality holding if $n\leq x$.  First we prove the following proposition

\begin{pro}
Let $1/2\leq \sigma \leq 1$. Let $T$ be large, and $x\geq (\log T)^2$ be a real number. Then for all positive integers $k$ with $x^k\leq T^{1/6}$, we have
$$\frac{1}{T}\int_{T}^{2T}|\zeta(\sigma+it)|^{2k}dt\geq \sum_{n\leq x^{k}}\frac{d_k(n;x)^2}{n^{2\sigma}}+O\left(T^{-1/6}\right).$$
\end{pro}

\begin{proof}
Let
$$\da(t):=\sum_{n\leq x}\frac{1}{n^{\sigma+it}}.$$
We shall evaluate the moments
$$ S_1:=\frac{1}{T}\int_{T}^{2T}\zeta(\sigma+it)\da(t)^{k-1}\overline{\da(t)}^kdt, \text{ and } S_2:=\frac{1}{T}\int_{T}^{2T}|\da(t)|^{2k}dt.$$
Let us begin with $S_2$. Since $k$ is a positive integer, then $\da(t)^k=\sum_{n\leq x^k}d_k(n;x)/n^{\sigma+it}$, which gives that
$$S_2=\sum_{m,n\leq x^k}\frac{d_k(n;x)d_k(m;x)}{(mn)^{\sigma}}\frac{1}{T}\int_{T}^{2T}\left(\frac{m}{n}\right)^{it}dt.$$
The contribution of the diagonal terms $m=n$ equals $\sum_{n\leq x^{k}}d_k(n;x)^2/n^{2\sigma}$. Moreover if $m\neq n$ then
 $\int_{T}^{2T}\left(\frac{m}{n}\right)^{it}dt\ll 1/|\log(m/n)|\ll T^{1/6}$, since  both $m$ and $n$ are below $T^{1/6}$. This along with (2.3) implies that the contribution of these terms is
$$\ll T^{-5/6}\left(\sum_{n\leq T^{1/6}}\frac{d_k(n)}{n^{\sigma}}\right)^2\ll T^{-1/2-\sigma/3}(\log T)^{2k+1}\leq T^{-1/2}.$$
This implies that
\begin{equation}
S_2= \sum_{n\leq x^{k}}\frac{d_k(n;x)^2}{n^{2\sigma}}+O\left(T^{-1/2}\right).
\end{equation}
Now we estimate $S_1$. Using the simple approximation (see Theorem 4.11 of [27]) $\zeta(\sigma+it)=\sum_{n\leq T}1/n^{\sigma+it}+O(T^{-\sigma}),$ we deduce that
\begin{equation}
 S_1=\sum_{n\leq T}\sum_{a,b\leq x^k}\frac{d_{k-1}(a;x)d_k(b;x)}{(abn)^{\sigma}}\frac{1}{T}\int_{T}^{2T}\left(\frac{b}{an}\right)^{it}dt+E_8,
 \end{equation}
where $E_8\ll T^{-\sigma} S_2\ll T^{-1/2}(\sum_{n\leq x^{k}}d_k(n)/n^{\sigma})^2\ll T^{-1/6}$, which follows from (2.3). The contribution of the diagonal terms $an=b$ to the main term on the RHS of (7.7) equals
\begin{equation}
 \sum_{b\leq x^k}\frac{d_k(b;x)}{b^{2\sigma}}\sum_{\substack{n\leq T, a\leq x^k\\ an=b}} d_{k-1}(a;x)\geq \sum_{b\leq x^k}\frac{d_k(b;x)^2}{b^{2\sigma}},
 \end{equation}
 since $$\sum_{\substack{ n\leq T, a\leq x^k\\ an=b}} d_{k-1}(a;x)\geq \sum_{\substack{ n\leq x, a\leq x^k\\ an=b}} d_{k-1}(a;x)=d_k(b;x).$$
Now we estimate the contribution of the off-diagonal terms $an\neq b$. First if $n\leq 2T^{1/6}$ then $an\leq 2T^{1/3}$ which  implies that $|\log(b/an)|\gg T^{-1/3}$. Moreover in this case we have that $\sum_{n<2T^{\frac{1}{6}}}n^{-\sigma}\ll T^{1/12}.$ Now if $n>2T^{1/6}$ then $an>2b$ which gives that $|\log(b/an)|\gg 1$, and in this case one has $\sum_{n<2T}n^{-\sigma}\ll T^{1/2}.$ Therefore combining both cases we deduce that $\sum_{n\leq T} n^{-\sigma} \int_{T}^{2T}\left(\frac{b}{an}\right)^{it}dt\ll T^{1/2}$. Thus the contribution of these terms to (7.7) is
$$ \ll T^{-1/2}\left(\sum_{a\leq T^{1/6}}\frac{d_k(a)}{a^{\sigma}}\right)^2\ll T^{-1/6},$$
using (2.3). Hence we deduce that
\begin{equation}
 S_1\geq \sum_{n\leq x^k}\frac{d_k(n;x)^2}{n^{2\sigma}} + O(T^{-1/6}).
 \end{equation}
Finally combining (7.6), (7.9) along with H\"older's inequality, we deduce that
$$ \frac{1}{T}\int_{T}^{2T}|\zeta(\sigma+it)|^{2k}dt\geq\frac{|S_1|^{2k}}{S_2^{2k-1}}\geq \sum_{n\leq x^k}\frac{d_k(n;x)^2}{n^{2\sigma}} + O(T^{-1/6}).$$

\end{proof}

Let $x=T^{1/(6k)}$. If $k$ is bounded it is not so hard to prove that
\begin{equation}
\sum_{n\leq x^k}\frac{d_k(n;x)^2}{n^{2\sigma}}= (1+o(1)) \sum_{n=1}^{\infty}\frac{d_k(n)^2}{n^{2\sigma}}, \text{ as } T\to \infty,
\end{equation}
for any $\sigma>1/2$. Our aim is to prove this asymptotic relation in a uniform range of $k$. To this end our idea consists of expressing  the sum  $\sum_{n\leq x^k}d_k(n;x)^2n^{-2\sigma}$ as the $2k$-th moment of a sum of certain random variables. Let
$\{X(p)\}_{p \text{ prime}}$ be a sequence of independent random variables uniformly distributed on the unit circle. We extend $\{X(p)\}_{p}$ to a completely multiplicative sequence of random variables $\{X(n)\}_{n\in \Bbb{N}}$ by defining $X(n):=\prod_iX(p_i)^{a_i}$ if $n=\prod_ip_i^{a_i}$. Then $\ex(X(n)\overline{X(m)})=1$ if $m=n$ and equals $0$ otherwise. As before set $\zeta(\sigma,X)=\lim_{N\to\infty}\sum_{n\leq N}X(n)/n^{\sigma}=\prod_{p}\left(1-X(p)p^{-\sigma}\right)^{-1},$
which is absolutely convergent almost surely for $\sigma>1/2$. Then Corollary 3.4 gives that
\begin{equation}
\ex(|\zeta(\sigma,X)|^{2k})=\sum_{n\geq 1}\frac{d_k(n)^2}{n^{2\sigma}}=\exp\left(G_1(\sigma)\frac{(2k)^{1/\sigma}}{\log k}\left(1+O\left(\frac{1}{\log k}\right)\right)\right).
\end{equation}
Moreover from the fact that $(\sum_{n\leq x} X(n)/n^{\sigma})^{k}=\sum_{n\leq x^{k}}d_k(n;x)X(n)/n^{\sigma}$, one can see that
$$ \ex\left(\left|\sum_{n\leq x}\frac{X(n)}{n^{\sigma}}\right|^{2k}\right)=\sum_{n,m\leq x^{k}}\frac{d_k(n;x)d_k(m;x)\ex(X(n)\overline{X(m)})}{(nm)^{\sigma}}=\sum_{n\leq x^k}\frac{d_k(n;x)^2}{n^{2\sigma}}.$$
Then the question of determining when does (7.10) hold is equivalent to understand when does the $2k$-th moment of the partial sum $\sum_{n\leq x}X(n)/n^{\sigma}$ approximate that of $\zeta(\sigma,X)$. We prove the following result which we combine with Proposition 7.3 to get Theorem 1.5.

\begin{pro} Let $T$ be large and put $x=T^{1/(6k)}$. Let $\epsilon>0$ be small and $0< \alpha<1$ be a real number. Then there exists a constant $c>0$ such that uniformly for any $1/2+1/(\log T)^{\alpha}<\sigma<1-\epsilon$ and all positive integers $k\leq c((2\sigma-1)\log T)^{\sigma}$, we have
$$\sum_{n\leq x^k}\frac{d_k(n;x)^2}{n^{2\sigma}}= \sum_{n=1}^{\infty}\frac{d_k(n)^2}{n^{2\sigma}}\left(1+O\left(\exp\left(-\frac{\log T}{10k\log_2 T}\right)\right)\right).$$
Moreover given $1/2<\sigma<1$, and a suitably large constant $C>0$ we have
$$ \sum_{n\leq x^k}\frac{d_k(n;x)^2}{n^{2\sigma}}=o\left(\sum_{n=1}^{\infty}\frac{d_k(n)^2}{n^{2\sigma}}\right),$$
for all integers $k\geq C(\log T\log_2 T)^{\sigma}.$
\end{pro}
>From this result we can observe that given $1/2<\sigma<1$, there is a transition for the asymptotic behavior of $\sum_{n\leq x^k}d_k(n;x)^2n^{-2\sigma}$ at $k\approx (\log T)^{\sigma}$, which may explain why our method does not give good lower bounds for the moments of $\zeta(\sigma+it)$ beyond that range of $k$.
\begin{proof} We begin by proving the first assertion. Notice that $\sum_{n\leq x^k}d_k(n;x)^2n^{-2\sigma}\leq \sum_{n=1}^{\infty}d_k(n)^2n^{-2\sigma}$. To prove the lower bound our idea consists of bounding the $2k$-th moment of the tail $\sum_{n>x}X(n)/n^{\sigma}$. We have
$$ \ex\left(\left|\sum_{n>x}\frac{X(n)}{n^{\sigma}}\right|^{2k}\right)= \sum_{n>x^{k}}\frac{f_k(n;x)^2}{n^{2\sigma}},$$
where $f_k(n;x)$ is the number of ways of writing $n$ as $a_1\cdots a_k$ with $a_i$ being positive integers such that $a_i>x$. Clearly $f_k(n;x)\leq d_k(n)$. Let $\alpha= 1/\log_2 T$. Then by Lemma 2.1 there exists a constant $C_0>0$ such that
\begin{equation}
\begin{aligned}
\ex&\left(\left|\zeta(\sigma,X)-\sum_{n\leq x}\frac{X(n)}{n^{\sigma}}\right|^{2k}\right)       =\ex\left(\left|\sum_{n>x}\frac{X(n)}{n^{\sigma}}\right|^{2k}\right)\leq \sum_{n>T^{1/6}}\frac{d_k(n)^2}{n^{2\sigma}}\\
&\leq T^{-\alpha/3}\sum_{n=1}^{\infty}\frac{d_k(n)^2}{n^{2\sigma-2\alpha}}\leq T^{-\alpha/3}\exp\left(C_0\frac{k^{1/(\sigma-\alpha)}}{(2\sigma-1)\log k}\right)\\
&\leq \exp\left(-\frac{\log T}{3\log_2 T}+ e^4C_0\frac{k^{1/\sigma}}{(2\sigma-1)\log k}\right)
\leq \exp\left(-\frac{\log T}{4\log_2 T}\right),
\end{aligned}
\end{equation}
by our hypothesis on $\sigma$ and $k$, if $c$ is suitably small. Furthermore by Minkowski's inequality we have that
$$ \ex\left(\left|\zeta(\sigma,X)\right|^{2k}\right)^{\frac{1}{2k}}\leq \ex\left(\left|\sum_{n\leq x}\frac{X(n)}{n^{\sigma}}\right|^{2k}\right)^{\frac{1}{2k}}+ \ex\left(\left|\zeta(\sigma,X)-\sum_{n\leq x}\frac{X(n)}{n^{\sigma}}\right|^{2k}\right)^{\frac{1}{2k}},$$
which in view of (7.11) and (7.12) implies that
$$
\ex\left(\left|\sum_{n\leq x}\frac{X(n)}{n^{\sigma}}\right|^{2k}\right)^{\frac{1}{2k}}\geq\left(\sum_{n\geq 1}\frac{d_k(n)^2}{n^{2\sigma}}\right)^{\frac{1}{2k}}-\exp\left(-\frac{\log T}{8k\log_2 T}\right).
$$
Since $\sum_{n\geq 1}d_k(n)^2n^{-2\sigma}\geq 1$ it follows that
$$
\sum_{n\leq x^k}\frac{d_k(n;x)^2}{n^{2\sigma}}\geq  \sum_{n=1}^{\infty}\frac{d_k(n)^2}{n^{2\sigma}}\left(1+O\left(\exp\left(-\frac{\log T}{10k\log_2 T}\right)\right)\right),
$$
which proves the first assertion of the proposition. Now observe that
$$ \sum_{n\leq x^k}\frac{d_k(n;x)^2}{n^{2\sigma}}=\ex\left(\left|\sum_{n\leq x} \frac{X(n)}{n^{\sigma}}\right|^{2k}\right)\leq \left(\sum_{n\leq x}\frac{1}{n^{\sigma}}\right)^{2k}\ll T^{1/2}.$$
Moreover, from (7.11) we can deduce that there exists a constant $C>0$ such that $\sum_{n=1}^{\infty}d_k(n)^2n^{-2\sigma}>T$, for all positive integers $k\geq C(\log T\log_2 T)^{\sigma}$, proving the second assertion of the proposition.
\end{proof}

\subsection*{Acknowledgments}

I am grateful to Andrew Granville and K. Soundararajan for suggesting the problem and for many valuable discussions.

\end{document}